\documentclass[12pt,letterpaper]{article}
\usepackage{color}
\usepackage{eqnarray}
\usepackage{amssymb,amsmath,amsthm}
\usepackage{subcaption}
\usepackage{multirow}
\usepackage{graphicx}
\usepackage{booktabs}
\usepackage{sectsty}
\usepackage[numbers,square]{natbib}
\sectionfont{\normalsize}
\subsectionfont{\normalsize}
\subsubsectionfont{\normalsize}
\paragraphfont{\normalsize}
\newtheorem{thm}{Theorem}
\newtheorem*{thm*}{Theorem}
\newtheorem{lem}[thm]{Lemma}

\newtheorem{cor}[thm]{Corollary}
\newtheorem{clm}[thm]{Claim}
\newcommand{\llll}[1] {\left #1}
\newcommand{\rrrr}[1] {\right #1}

\newcommand{\dddd}[2]{\dfrac{#1}{#2}}

\newcommand{\aaaa}{\alpha}

\newcommand{\ssss}{\sigma}

\newcommand{\dddddd}{\delta}
\newcommand{\llllll}{\lambda}
\newcommand{\bbbb}{\beta}
\newcommand{\GGGG}{\Gamma}
\newcommand{\gggg}{\gamma}
\newcommand{\oooo}{\omega}
\newcommand{\zzzz}{\zeta}

\begin{document}
\nocite{*}
\title{{\bf \normalsize 	APPROXIMATIONS FOR THE CAPUTO DERIVATIVE (I)}}
\author{Yuri Dimitrov\\
Department of Applied Mathematics and Statistics \\
University of Rousse, Rousse  7017, Bulgaria\\
\texttt{ymdimitrov@uni-ruse.bg}}
\maketitle
\begin{abstract} 
In this paper we construct approximations for the Caputo derivative of order $1-\aaaa,2-\aaaa,2$ and $3-\aaaa$. The approximations have weights  $0.5\llll((k+1)^{-\aaaa}-(k-1)^{-\aaaa}\rrrr)/\GGGG(1-\aaaa)$ and $k^{-1-\aaaa}/\GGGG(-\aaaa)$, and the higher accuracy is achieved by modifying the initial and last weights using the expansion formulas for the left and right endpoints.   The approximations are applied for computing the numerical solution of ordinary fractional differential equations. The properties of the weights of the approximations of order $2-\aaaa$ are similar to the properties of the $L1$ approximation. In all experiments presented in the paper the accuracy of the numerical solutions  using the approximation of order $2-\aaaa$ which has weights $k^{-1-\aaaa}/\GGGG(-\aaaa)$ is higher than the accuracy of the numerical solutions using the $L1$ approximation for the Caputo derivative.

\noindent
{\bf 2010 Math Subject Classification:} 26A33, 34A08, 34E05, 41A25
\end{abstract}
\section{Introduction}
The Gr\"unwald formula approximation and the $L1$ approximation of the Caputo derivative have been regularly used for numerical solution of fractional differential equations \cite{JinLazarovZhou2016,LinXu2007,Ma2014,Murio2008,Podlubny1999}. The Gr\"unwald formula approximation has weights $(-1)^k\binom{\aaaa}{k}$ and accuracy $O(h)$. The $L1$ approximation \eqref{3_10} has order $2-\aaaa$ and weights $\ssss_k^{(\aaaa)}=\llll((k-1)^{1-\aaaa}-2k^{1-\aaaa}+(k+1)^{1-\aaaa}\rrrr)/\GGGG(2-\aaaa)$. The  weights of the $L1$ approximation are linear combinations of terms which have  power $1-\aaaa$. In the present paper we construct approximations of the Caputo derivative whose weights consist of terms which have power $-\aaaa$ and $-1-\aaaa$. The accuracy of the numerical solution of order $2-\aaaa$ is influenced by the coefficient of the term $y''(x)h^{2-\aaaa}$ in the expansion of the approximation. In Table 1, Table 3 and Table 9 we compute the error and the order of the numerical solutions of the fractional relaxation equation which use approximations \eqref{3_10},\eqref{5_30} and \eqref{6_20} of the Caputo derivative. In all experiments the accuracy of the numerical solutions using approximation \eqref{5_30} is higher than the accuracy of the numerical solutions using approximations \eqref{3_10} and \eqref{6_20}.
The Caputo derivative  of order $\aaaa$, where $0<\aaaa<1$ is defined as
$$y^{(\aaaa)}(x)=D^{\aaaa} y(x)=\dddd{1}{\Gamma (1-\aaaa)}\int_0^x \dfrac{y'(t)}{(x-t)^{\aaaa}}d t.$$
When the first derivative of the function $y$ is a bounded function the Caputo derivative at the initial point $x=0$ is equal to zero. The exponential function has Caputo derivative 
$D^\aaaa e^{\llllll x}=\llllll x^{1-\aaaa}E_{1,2-\aaaa}(\llllll x)$
and
$$D^\aaaa \cos (\llllll x)=\dddd{i \llllll}{2} x^{1-\aaaa}
\llll(E_{1,2-\aaaa}(i \llllll x)-E_{1,2-\aaaa}(-i\llllll x)\rrrr),$$
where $E_{1,2-\aaaa}( x)$ is the Mittag-Leffler function ($\bbbb=2-\aaaa$)
$$E_{\aaaa,\bbbb}(x)=\sum_{k=0}^\infty \dddd{x^k}{\GGGG(\aaaa k+\bbbb)}.$$
The analytical solutions of ordinary and partial differential equations are often expressed with the Mittag-Leffler function. The ordinary fractional differential equation
\begin{equation*} 
y^{(\aaaa)}(x)+D y(x)=F(x),\; y(0)=y_0,
\end{equation*}
has exact solution
$$y(x)= y_0 E_{\aaaa}(-D x^\aaaa)+\int_{0}^{x}\xi^{\aaaa-1} E_{\aaaa,\aaaa}\llll(-D \xi^\aaaa\rrrr)F(x-\xi)d\xi.$$
The finite-difference schemes for numerical solution of ordinary and partial fractional differential equations use an approximation for the fractional derivative.  Let $h=x/n$ and $y_n=y(n h)$ be the value of the function $y$ at the point $x_n=n h$. The  $L1$ approximation of the Caputo derivative is constructed by dividing the interval $[0,x]$ to subintervals of equal length $h$ and approximating the first derivative on each subinterval using a second-order central difference approximation.
\begin{equation}\label{3_10}
y^{(\alpha)}_n  =\dfrac{1}{\GGGG(2-\aaaa)h^\alpha}\sum_{k=0}^{n} \ssss_k^{(\alpha)} y_{n-k}+O\llll(h^{2-\aaaa}\rrrr),
\end{equation}
where $\ssss_0^{(\alpha)}=1$,  $\ssss_n^{(\alpha)}=(n-1)^{1-a}-n^{1-a}$ and
$$\ssss_k^{(\alpha)}=(k+1)^{1-\alpha}-2k^{1-\alpha}+(k-1)^{1-\alpha}, \quad (k=2,\cdots,n-1).$$
  Approximation \eqref{3_10} has accuracy $O(h^{2-\alpha})$ when $y\in C^2[0,x]$ (\cite{LinXu2007}). 	 The numbers $\ssss_k^{(\alpha)}$ have properties 
\begin{align}\label{3_20}
&\ssss_0^{(\alpha)}>0,\; \ssss_1^{(\alpha)}<\ssss_2^{(\alpha)}<\cdots<\ssss_k^{(\alpha)}<\cdots<\ssss_{n-1}^{(\alpha)}<0,\;\ssss_{n}^{(\alpha)}<0,\nonumber\\
&\sum_{k= 0}^n \ssss_k^{(\alpha)} = 0,\quad \sum_{k= 1}^n k \ssss_k^{(\alpha)} = -n^{1-\aaaa},\\
&\ssss_k^{(\alpha)}= \dddd{C_1}{k^{1+\aaaa}}+O\llll(\dddd{1}{k^{2+\aaaa}} \rrrr),\quad \ssss_n^{(\alpha)}= \dddd{C_2}{n^\aaaa}+O\llll(\dddd{1}{n^{1+\aaaa}} \rrrr),\nonumber
\end{align}
where $C_1=\aaaa(\aaaa-1)$ and $C_2=\aaaa-1$. 
{ \renewcommand{\arraystretch}{1.1}
\begin{table}[h]
	\caption{Error and order of numerical solution $NS[1]$ of Equation I with $\aaaa=0.25$, Equation II with $\aaaa=0.5$ and Equation III with $\aaaa=0.75$.}
	\small
	\centering
  \begin{tabular}{ l | c  c | c  c | c  c }
		\hline
		\hline
		\multirow{2}*{ $\quad \boldsymbol{h}$}  & \multicolumn{2}{c|}{{\bf Equation I}} & \multicolumn{2}{c|}{{\bf Equation II}}  & \multicolumn{2}{c}{{\bf Equation III}} \\
		\cline{2-7}  
   & $Error$ & $Order$  & $Error$ & $Order$  & $Error$ & $Order$ \\ 
		\hline \hline
$0.003125$    & $0.0000466$          & $1.6970$  & $0.0000513$           & $1.4857$    & $0.0024184$  & $1.2442$       \\ 
$0.0015625$   & $0.0000143$          & $1.7071$  & $0.0000183$           & $1.4901$    & $0.0010191$  & $1.2468$       \\ 
$0.00078125$  & $4.3\times 10^{-6}$  & $1.7150$  & $6.5\times 10^{-6}$   & $1.4931$    & $0.0004290$  & $1.2482$        \\ 
$0.000390625$ & $1.3\times 10^{-6}$  & $1.7212$  & $2.3\times 10^{-6}$   & $1.4952$    & $0.0001805$  & $1.2490$        \\
\hline
  \end{tabular}
	\end{table}
	}	
	
	 In section 4 we derive the numerical solution $NS[*]$ of the fractional relaxation equation which uses approximation \eqref{20_10} for the Caputo derivative.  In Table 1 we compute the error and the order of numerical solution $NS[1]$ of Equation I and $\aaaa=0.25$, Equation II and $\aaaa=0.5$ and Equation III with $\aaaa=0.75$. In Theorem 9 and Theorem 10  we prove that the numerical solution of the fractional relaxation equation converges to the exact solution with accuracy $O\llll(h^{2-\aaaa}\rrrr)$  when the coefficient $D$ is positive or $D$ is a bounded negative number and the approximation of the Caputo derivative has order $2-\aaaa$ and satisfies  \eqref{3_20}. In section 2 and section 5 we obtain approximations \eqref{5_30} and \eqref{6_20} of the Caputo derivative. Approximations \eqref{3_10},\eqref{5_30} and \eqref{6_20} have order $2-\aaaa$ and satisfy properties \eqref{3_20}. The asymptotic expansions of the  integral approximations for the fractional derivatives involve the values of the  {\it  Riemann zeta function} defined as
		\begin{equation*}
	\zzzz(\aaaa)=
	\displaystyle{\sum_{n=1}^\infty \dddd{1}{n^\aaaa}}, \;(\aaaa>1),
	\quad	\zzzz(\aaaa)=\displaystyle{\dfrac{1}{1-2^{1-\aaaa}}\sum_{n=1}^\infty \dddd{(-1)^{n-1}}{n^\aaaa}}, \;  (\aaaa>0).  
	\end{equation*}
	When $\aaaa<0$, the values of the Riemann zeta function are computed from the the functional equation
$$\zzzz(\aaaa)=2^\aaaa\pi^{\aaaa-1}\sin \llll(\dddd{\pi \aaaa}{2} \rrrr)\GGGG (1-\aaaa)\zzzz (1-\aaaa).$$
Another important special function in fractional calculus is the {\it polylogarithm function} is defined as
$$Li_\aaaa(x)=\sum_{n=1}^{\infty}\dddd{x^n}{n^\aaaa}=x+\dddd{x^2}{2^\aaaa}+\cdots+\dddd{x^n}{n^\aaaa}+\cdots$$
The polylogarithm function satisfies $Li_\aaaa(1)=\zzzz(\aaaa)$ and has  expansion
\begin{align}\label{4_10}
Li_\aaaa(x)=\GGGG(1-\aaaa)\llll(\ln \dddd{1}{x}\rrrr)^{\aaaa-1}+\sum_{n=0}^{\infty}\dddd{\zzzz(\aaaa-n)}{n!}\llll(\ln x\rrrr)^n,
\end{align}
where $\aaaa \neq 1,2,3,\cdots$ and $|\ln x|<2\pi$. 
In \cite{Dimitrov2015}  we obtained the  second-order expansion formula of the $L1$ approximation 
\begin{align}\label{4_15}
\dddd{1}{\GGGG(2-\aaaa)h^\aaaa}\sum_{k=0}^n \ssss_k^{(\aaaa)} y(x-kh)=y^{(\aaaa)}(x)+&\dddd{\zzzz(\aaaa-1)}{\GGGG(2-\aaaa)}y''(x)h^{2-\aaaa}+O\llll(h^2\rrrr),
\end{align}
and the second-order approximation of the Caputo derivative
\begin{equation}\label{4_20}
y^{(\aaaa)}_n=\dddd{1}{\GGGG(2-\aaaa)h^\aaaa}\sum_{k=0}^n \dddddd_k^{(\aaaa)} y_{n-k}+O\llll(h^{2}\rrrr),
\end{equation}
where $\dddddd_k^{(\aaaa)}=\ssss_k^{(\aaaa)}$ for $2\leq k\leq n$ and
$$\dddddd_0^{(\aaaa)}=\ssss_0^{(\aaaa)}-\zzzz(\aaaa-1),\; \dddddd_1^{(\aaaa)}=\ssss_1^{(\aaaa)}+2\zzzz(\aaaa-1),\; \dddddd_2^{(\aaaa)}=\ssss_2^{(\aaaa)}-\zzzz(\aaaa-1).$$
The numbers $\dddddd_k^{(\alpha)}$ satisfy
\begin{equation}\label{4_30}
\dddddd_0^{(\alpha)}>0,\dddddd_1^{(\alpha)}<0,\dddddd_2^{(\alpha)}>0, \dddddd_3^{(\alpha)}<\cdots<\dddddd_k^{(\alpha)}<\cdots<\dddddd_{n-1}^{(\alpha)}<0,\dddddd_{n}^{(\alpha)}<0.
\end{equation}

	The fractional integral of order $\aaaa>0$ on the interval $[0,x]$ is defined as 
$$I^\aaaa y(x)=\dddd{1}{\GGGG(\aaaa)}\int_0^x (x-t)^{\aaaa-1} y(t) d t.
$$
 Denote by $ K^{\aaaa} y(x)$ the fractional integral $  K^{\aaaa} y(x)=\GGGG(\aaaa) I^{\aaaa} y(x)$. In \cite{Dimitrov2016} we use the Fourier transform method and the asymptotic expansion \eqref{4_10} of the polylogarithm function to derive the asymptotic formula for the Riemann sum approximation of the fractional integral $ K^{\aaaa} y(x)$ at the right endpoint.
\begin{align}  h^{\aaaa} \sum_{k=1}^{n-1} &\dddd{y(x-k h)}{k^{1-\aaaa}}=\int_0^x \dddd{y(t)}{(x-t)^{1-\aaaa}}dt +\sum_{k=0}^\infty (-1)^k\dddd{\zzzz(1-\aaaa-k)}{k!}y^{(k)}(x) h^{k+\aaaa}\nonumber\\
&-\GGGG(\aaaa)\sum_{k=0}^\infty \dddd{B_{k+1}}{(k+1)!}\llll( \sum_{m=0}^k (-1)^m\binom{k}{m}\dddd{x^{\aaaa-m-1}}{\GGGG(\aaaa-m)}y^{(m-k)}(0) \rrrr)h^{k+1}.\label{5_10}
\end{align}
The asymptotic expansion formula for the left endpoint of \eqref{5_10} is obtained from the Euler-Maclaurin formula for the function $z(t)=y(t)(x-t)^{\aaaa-1}$.
In section 2, we use \eqref{5_10} to obtain the approximation of the Caputo derivative
\begin{align}\label{5_20}
\dddd{1}{2h^\aaaa}\Bigg(y_n+&\dddd{y_{n-1}}{2^\aaaa} +\sum_{k=2}^{n-2} y_{n-k}\llll(\dddd{1}{(k+1)^\aaaa}-\dddd{1}{(k-1)^\aaaa}  \rrrr)-\dddd{y_{1}}{(n-2)^\aaaa} -\dddd{y_{0}}{(n-1)^\aaaa} \Bigg)\nonumber\\
&=\GGGG(1-\aaaa)y_n^{(\aaaa)}+\zzzz(\aaaa)y'_n h^{1-\aaaa}-\zzzz(\aaaa-1)y''_n h^{2-\aaaa}+O\llll( h^{2} \rrrr).
\end{align}
Approximation \eqref{5_20} has second order accuracy when the function $y$ satisfies $y'(0)=y''(0)=0$. Denote
$$S_n[\aaaa]=\sum_{k=1}^{n-1}\dddd{1}{k^{\aaaa}}-\zzzz(\aaaa).$$
From \eqref{5_20} we obtain approximations \eqref{5_30} and \eqref{6_05} of the Caputo derivative
\begin{align}\label{5_30} 
\dddd{1}{2\GGGG(1-\aaaa)h^\aaaa}\sum_{k=0}^{n}\ssss_k^{(\aaaa)} y_{n-k}=y_n^{(\aaaa)}+O\llll( h^{2-\aaaa} \rrrr),
\end{align}
where
$\ssss_0^{(\aaaa)}=1-2\zzzz(\aaaa),\quad \ssss_1^{(\aaaa)}=\dddd{1}{2^\aaaa}+2\zzzz(\aaaa), $
$$\ssss_k^{(\aaaa)}=\dddd{1}{(k+1)^\aaaa}-\dddd{1}{(k-1)^\aaaa},\qquad (k=2,\cdots,n-2),$$
$$\ssss_{n-1}^{(\aaaa)}=-\dddd{1}{(n-2)^\aaaa}-2 S_n[\aaaa]+\dddd{2n^{1-\aaaa}}{1-\aaaa},\; \ssss_{n}^{(\aaaa)}=-\dddd{1}{(n-1)^\aaaa}+2 S_n[\aaaa]-\dddd{2n^{1-\aaaa}}{1-\aaaa}.$$
\begin{align}\label{6_05}
\dddd{1}{2\GGGG(1-\aaaa)h^\aaaa}\sum_{k=0}^{n}\dddddd_k^{(\aaaa)} y_{n-k}=y_n^{(\aaaa)}+O\llll( h^2 \rrrr),
\end{align}
where $\dddddd_0^{(\aaaa)}=1-3 \zzzz(\aaaa)+2 \zzzz(\aaaa-1),\; \dddddd_1^{(\aaaa)}=\dddd{1}{2^\aaaa}+4\zzzz(\aaaa)-4 \zzzz(\aaaa-1),$
\begin{align*}
&\dddddd_2^{(\aaaa)}=\dddd{1}{3^\aaaa}-1-\zzzz(\aaaa)+2\zzzz(\aaaa-1),\;\dddddd_{k}^{(\aaaa)}=\ssss_{k}^{(\aaaa)}, \quad (k=3,\cdots,n).
\end{align*}
In section 5 we use the Fourier transform method and the Euler-Maclaurin formula
for the function $y(t)/(x-t)^{1+\aaaa}$ to derive the  expansion formula
	\begin{align}\label{6_10}
\dddd{1}{h^\aaaa}\sum_{k=1}^{n-1}\dddd{y(x-kh)}{k^{1+\aaaa}}=&\GGGG(-\aaaa)y^{(\aaaa)}(x)+\dddd{ y(0)}{\aaaa x^{\aaaa}}-\dddd{y(0) h}{2 x^{1+\aaaa}}\nonumber\\
&+\sum_{m=0}^\infty\dddd{(-1)^m}{m!}\zzzz(1+\aaaa-m)y^{(m)}(x)h^{m-\aaaa}\\
-\sum_{m=1}^\infty \dddd{B_{2m}}{(2m)!}&\llll( \sum_{k=0}^{2m-1}\binom{2m-1}{k}\dddd{\GGGG(1+\aaaa+k)}{\GGGG(1+\aaaa)}\dddd{y^{(2m-1-k)}(0)}{x^{1+\aaaa+k}}\rrrr)h^{2m}.\nonumber
\end{align}
From \eqref{6_10} we obtain the approximations of the Caputo derivative \eqref{6_20} and \eqref{7_10} of order $2-\aaaa$ and $3-\aaaa$.
\begin{equation}\label{6_20}
\dddd{1}{\GGGG(-\aaaa)h^\aaaa}\sum_{k=0}^{n}\ssss_k^{(\aaaa)} y_{n-k}= y_n^{(\aaaa)}+O\llll(h^{2-\aaaa}\rrrr),
\end{equation}
where $\ssss_0^{(\aaaa)}=\zzzz(\aaaa)-\zzzz(1+\aaaa),\;\ssss_1^{(\aaaa)}=1-\zzzz(\aaaa),$
$$\ssss_k^{(\aaaa)}=\dddd{1}{k^{1+\aaaa}},\quad (k=2,\cdots,n-2),$$
$$\ssss_{n-1}^{(\aaaa)}=\dddd{1}{(n-1)^{1+\aaaa}}-n S_n [\aaaa+1]+S_n[\aaaa]-\dddd{n^{1-\aaaa}}{\aaaa(1-\aaaa)},$$
$$\ssss_{n}^{(\aaaa)}=(n-1)S_n[\aaaa+1]-S_n[\aaaa]+\dddd{n^{1-\aaaa}}{\aaaa(1-\aaaa)}.$$
\begin{equation}\label{7_10}
\dddd{1}{\GGGG(-\aaaa)h^\aaaa}\sum_{k=0}^{n}\dddddd_k^{(\aaaa)} y_{n-k}= y_n^{(\aaaa)}+O\llll(h^{3-\aaaa}\rrrr),
\end{equation}
where
$$\dddddd_0^{(\aaaa)}=-\zzzz(1+\aaaa)+\dddd{3}{2}\zzzz(\aaaa)-\dddd{1}{2}\zzzz(\aaaa-1),\;\dddddd_1^{(\aaaa)}=1-2\zzzz(\aaaa)+\zzzz(\aaaa-1),$$
$$\dddddd_2^{(\aaaa)}=\dddd{1}{2^{1+\aaaa}}+\dddd{1}{2}\zzzz(\aaaa)-\dddd{1}{2}\zzzz(\aaaa-1),\;\dddddd_{k}^{(\aaaa)}=\dddd{1}{k^{1+\aaaa}}\quad (k=3,\cdots,n-3),$$
\begin{align*}
\dddddd_{n-2}^{(\aaaa)}=\dddd{1}{(n-2)^{1+\aaaa}}-\dddd{1}{2} \Bigg(n(n-1)&S_n[\aaaa+1]-(2 n-1)S_n[\aaaa],\\
&+S_n[\aaaa-1]+\frac{(\aaaa+2 n-2) n^{1-\aaaa}}{(\aaaa-2) (\aaaa-1) \aaaa}\Bigg),
\end{align*}
\begin{align*}
\dddddd_{n-1}^{(\aaaa)}=\dddd{1}{(n-1)^{1+\aaaa}}+n(n-2) &S_n[\aaaa+1]-2(n-1)S_n[\aaaa]\\
&+S_n[\aaaa-1]+\dddd{2 (\aaaa+n-2) n^{1-\aaaa}}{\aaaa(1-\aaaa)(2-\aaaa)},
\end{align*}
\begin{align*}
\dddddd_n^{(\aaaa)}=-\dddd{1}{2} \Bigg((n-1) (n-2)&S_n[\aaaa+1]-(2 n-3)S_n[\aaaa] \\
&+S_n[\aaaa-1]+\dddd{(3 \aaaa+2 n-6) n^{1-\aaaa}}{\aaaa(1-\aaaa)(2-\aaaa)}\Bigg).
\end{align*}
 The expansion formula \eqref{4_15} of the $L1$ approximation contains the term $C_1(\aaaa)y''_n h^{2-\aaaa}$ where $C_1(\aaaa)=\zzzz(\aaaa-1)/\GGGG(2-\aaaa)$. The expansions of approximations \eqref{5_30} and \eqref{6_20} contain terms $C_9(\aaaa)y''_n h^{2-\aaaa}$ and $C_{12}(\aaaa)y''_n h^{2-\aaaa}$ with coefficients
$$C_9(\aaaa)=\frac{\zeta (\aaaa)-2\zeta (\aaaa-1)}{2\Gamma (1-\aaaa)},\; C_{12}(\aaaa)=\frac{\zeta (\aaaa)-\zeta (\aaaa-1)}{2\Gamma (-\aaaa)}.$$
		\begin{figure}[ht]
  \centering
  \caption{Graph of the coefficients $C_1(\aaaa)$(black), $C_9(\aaaa)$(red), $C_{12}(\aaaa)$(blue).} 
  \includegraphics[width=0.5\textwidth]{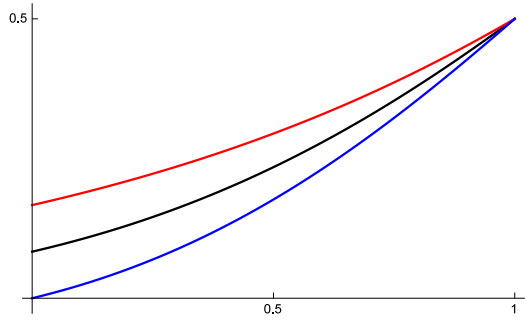}
\end{figure}

The accuracy of the  numerical solutions of fractional differential equations which use approximations \eqref{3_10}, \eqref{5_30} and \eqref{6_20} of the Caputo derivative is influenced by the values of the coefficients $C_1(\aaaa),C_9(\aaaa)$ and $C_{12}(\aaaa)$. In Figure 1 we compare the coefficients $C_1(\aaaa)$, $C_9(\aaaa)$  and $C_{12}(\aaaa)$. The three coefficients are positive and the coefficient $C_{12}(\aaaa)$ of approximation \eqref{6_20} is smaller than $C_1(\aaaa)$ and $C_9(\aaaa)$.  In Table 1 and Table 9 we compute the maximum errors of the numerical solutions of Equation I, Equation II and Equation III which use approximations  \eqref{3_10} and \eqref{6_20}. The accuracy of the numerical solution which uses approximation \eqref{6_20} for the Caputo derivative is higher  than the  accuracy of the numerical solution using the $L1$ approximation \eqref{3_10}. The improvement is  $88\%$ for Equation I and $\aaaa=0.25$, $35\%$ for Equation II and $\aaaa=0.5$ and $9\%$ for Equation III with $\aaaa=0.75$.
\section{Approximation for the Caputo Derivative of Order $\boldsymbol{2-\aaaa}$}
In this section we use approximation \eqref{5_10} for the fractional integral  to obtain approximations for the Caputo derivative of order $1-\aaaa$ and $2-\aaaa$. From \eqref{5_10}
\begin{align}\label{8_10}
h^{\bbbb}\sum_{k=1}^{n-1} \dddd{ z_{n-k}}{k^{1-\bbbb}}=K^{\bbbb}z_n+\zzzz(1-\bbbb) z_n h^{\aaaa}-\zzzz(-\bbbb)z'_n h^{1+\bbbb}
+O\llll(h^{2+\bbbb} \rrrr).
\end{align}
Approximation \eqref{8_10} has accuracy $O\llll(h^{2+\bbbb}\rrrr)$ when the function $z$ satisfies  $z(0)=z'(0)=0$. The Caputo derivative of order $\aaaa$ satisfies 
$$\GGGG(1-\aaaa) y_n^{(\aaaa)}=K^{1-\aaaa}y'_n.$$
 From \eqref{8_10} with $\bbbb=1-\aaaa$ and $z(x)=y'(x)$ we obtain
\begin{align}\label{8_20}
h^{1-\aaaa}\sum_{k=1}^{n-1} \dddd{ y'_{n-k}}{k^{\aaaa}}=\GGGG(1-\aaaa)y^{(\aaaa)}_n+\zzzz(\aaaa) y'_n h^{1-\aaaa}-\zzzz(\aaaa-1)y''_nh^{2-\aaaa}
+O\llll(h^{3-\aaaa} \rrrr).
\end{align}
By approximating $y'_{n-k}$ in \eqref{8_20}  using central difference approximation
\begin{align}\label{9_10}
y'_{n-k}=\dddd{y_{n-k+1}-y_{n-k-1}}{2h}+\dddd{h^2}{6}y'''(\xi_{n-k}), 
\end{align}
where $x_{n-k-1}<\xi_{n-k}<x_{n-k+1}$, we obtain the approximation
\begin{align}\label{9_20}
&h^{1-\aaaa}\sum_{k=1}^{n-1} \dddd{ y'_{n-k}}{k^{\aaaa}}\approx \dddd{1}{2h^\aaaa} \sum_{k=1}^{n-1} \dddd{ y_{n-k+1}-y_{n-k-1}}{k^{\aaaa}}=\dddd{1}{2h^\aaaa}\llll( \sum_{k=1}^{n-1} \dddd{ y_{n-k+1}}{k^{\aaaa}}-\sum_{k=1}^{n-1} \dddd{y_{n-k-1}}{k^{\aaaa}}\rrrr)\nonumber\\
&=\dddd{1}{2h^\aaaa}\Bigg(y_n+\dddd{y_{n-1}}{2^\aaaa}+ \sum_{k=3}^{n-1} \dddd{ y_{n-k+1}}{k^{\aaaa}}-\sum_{k=1}^{n-3} \dddd{y_{n-k-1}}{k^{\aaaa}}-\dddd{y_{1}}{(n-2)^\aaaa} -\dddd{y_{0}}{(n-1)^\aaaa}\Bigg)\\
&=\dddd{1}{2h^\aaaa}\llll(y_n+\dddd{y_{n-1}}{2^\aaaa} +\sum_{k=2}^{n-2} y_{n-k}\llll(\dddd{1}{(k+1)^\aaaa}-\dddd{1}{(k-1)^\aaaa}  \rrrr)-\dddd{y_{1}}{(n-2)^\aaaa} -\dddd{y_{0}}{(n-1)^\aaaa} \rrrr).\nonumber
\end{align}
In Lemma 1 we show that approximation \eqref{9_20} has second order accuracy.
\begin{lem} Let $M_3=\max_{0\leq t\leq x}|y'''(t)|$. Then
$$\llll| h^{1-\aaaa}\sum_{k=1}^{n-1} \dddd{ y'_{n-k}}{k^{\aaaa}}- \dddd{1}{2h^\aaaa} \sum_{k=1}^{n-1} \dddd{ y_{n-k+1}-y_{n-k-1}}{k^{\aaaa}}\rrrr|
<\llll(\dddd{M_3 x^{1-\aaaa}}{6(1-\aaaa)}\rrrr)h^{2}.$$
\end{lem}
\begin{proof} From \eqref{9_10}
\begin{align*}
\Bigg| h^{1-\aaaa}\sum_{k=1}^{n-1}\dddd{ y'_{n-k}}{k^{\aaaa}}-\dddd{1}{2h^\aaaa} \sum_{k=1}^{n-1} &\dddd{ y_{n-k+1}-y_{n-k-1}}{k^{\aaaa}}\Bigg|=\dddd{h^{3-\aaaa}}{6}\llll|\sum_{k=1}^{n-1} \dddd{y'''(\xi_{n-k})}{k^\aaaa}  \rrrr|\\
&<\dddd{h^{3-\aaaa}M_3}{6}\sum_{k=1}^{n-1} \dddd{1}{k^\aaaa}<\dddd{h^{3-\aaaa}M_3}{6}\int_0^n \dddd{1}{t^\aaaa}d t\\
&<\dddd{h^{3-\aaaa}M_3 n^{1-\aaaa}}{6(1-\aaaa)}=\dddd{M_3 x^{1-\aaaa}h^{2}}{6(1-\aaaa)}.
\end{align*}
\end{proof}
Let
\begin{align}\label{9_30}
\mathcal{A}_n^\oooo[y]=\dddd{1}{2\GGGG(1-\aaaa)h^\aaaa}\sum_{k=0}^{n}\oooo_k^{(\aaaa)} y_{n-k},
\end{align}
where
$\oooo_0^{(\aaaa)}=1,\quad \oooo_1^{(\aaaa)}=\dddd{1}{2^\aaaa},\quad \oooo_{n-1}^{(\aaaa)}=-\dddd{1}{(n-2)^\aaaa},\quad \oooo_{n}^{(\aaaa)}=-\dddd{1}{(n-1)^\aaaa}$,
$$\oooo_k^{(\aaaa)}=\dddd{1}{(k+1)^\aaaa}-\dddd{1}{(k-1)^\aaaa},\quad (k=2,\cdots,n-2).$$
From \eqref{8_20} and \eqref{9_20} we obtain the expansion formula of approximation $\mathcal{A}_n^\oooo[y]$.
	\begin{lem} Let $y(0)=y'(0)=0$. Then 
\begin{align}\label{10_10}
\mathcal{A}_n^\oooo[y]=y_n^{(\aaaa)}+\dddd{\zzzz(\aaaa)}{\GGGG(1-\aaaa)}y'_n h^{1-\aaaa}-\dddd{\zzzz(\aaaa-1)}{\GGGG(1-\aaaa)}y''_n h^{2-\aaaa}+O\llll( h^{2} \rrrr).
\end{align}
	\end{lem}
	The numbers $\oooo_k^{(\aaaa)}$ satisfy $\sum_{k=0}^{n}\oooo_k^{(\aaaa)}=0$ and  $\mathcal{A}_n^\oooo[1]=D^\aaaa 1=0$. From Lemma 2, approximation $\mathcal{A}_n^\oooo[y]$ for the Caputo derivative has order $1-\aaaa$.
\begin{align}\label{10_20}
\dddd{1}{2\GGGG(1-\aaaa)h^\aaaa}\sum_{k=0}^{n}\oooo_k^{(\aaaa)} y_{n-k}=y_n^{(\aaaa)}+O\llll( h^{1-\aaaa} \rrrr).
\end{align}
 When $n=2$, approximation $\mathcal{A}_2^\oooo[y]$ has weights $\oooo_0^{(\aaaa)}=1,\oooo_1^{(\aaaa)}=0,\oooo_2^{(\aaaa)}=-1$. Denote by Equation I, Equation II and Equation III the fractional relaxation equations 
\begin{align}\tag{{\bf Equation I}}
y^{(\aaaa)}(x)+y(x)=1&+x+x^2+x^3+x^4+\dddd{x^{1-\aaaa}}{\GGGG(2-\aaaa)}\nonumber\\
&+\dddd{2 x^{2-\aaaa}}{\GGGG(3-\aaaa)}
+\dddd{6 x^{3-\aaaa}}{\GGGG(4-\aaaa)}+\dddd{24 x^{4-\aaaa}}{\GGGG(5-\aaaa)},\; y(0)=1.\nonumber
\end{align}
Equation I has the solution $y(x)=1+x+x^2+x^3+x^4$.		
\begin{align} \tag{{\bf Equation II}}
y^{(\aaaa)}(x)+y(x)=e^x+x^{1-\aaaa}E_{1,2-\aaaa}(x),\quad y(0)=1.
\end{align}
Equation II has the solution $y(x)=e^x$.
\begin{align} 
&y^{(\aaaa)}(x)+y(x)=\cos (2\pi x)+i \pi x^{1-\aaaa}
\llll(E_{1,2-\aaaa}(2\pi i x)-E_{1,2-\aaaa}(-2\pi i x)\rrrr),\nonumber\\
& y(0)=1. \tag{{\bf Equation III}}\nonumber
\end{align}
Equation III has the solution $y(x)=\cos (2\pi x)$. The numerical results for the maximum error and order of numerical solution $NS[20]$ of Equation I, Equation II and Equation III are given in Table 2.
\setlength{\tabcolsep}{0.5em}
{ \renewcommand{\arraystretch}{1.1}
\begin{table}[ht]
	\caption{Error and order of numerical solution $NS[20]$ of Equation I with $\aaaa=0.25$, Equation II with $\aaaa=0.5$ and Equation III with $\aaaa=0.75$.}
	\small
	\centering
  \begin{tabular}{ l | c  c | c  c | c  c }
		\hline
		\hline
		\multirow{2}*{ $\quad \boldsymbol{h}$}  & \multicolumn{2}{c|}{{\bf Equation I}} & \multicolumn{2}{c|}{{\bf Equation II}}  & \multicolumn{2}{c}{{\bf Equation III}} \\
		\cline{2-7}  
   & $Error$ & $Order$  & $Error$ & $Order$  & $Error$ & $Order$ \\ 
		\hline \hline
$0.003125$    & $0.040435$  & $0.7606$  & $0.061232$   & $0.5206$    & $0.521231$   & $0.3076$       \\ 
$0.0015625$   & $0.023913$  & $0.7578$  & $0.042851$   & $0.5150$    & $0.423886$   & $0.2983$       \\ 
$0.00078125$  & $0.014162$  & $0.7558$  & $0.030075$   & $0.5108$    & $0.346614$   & $0.2904$        \\ 
$0.000390625$ & $0.008395$  & $0.7544$  & $0.021152$   & $0.5077$    & $0.284739$   & $0.2837$        \\
\hline
  \end{tabular}
	\end{table}
	}
	
	By substituting $y'_n$ in \eqref{10_10} using the approximation
$$y'_{n}=\dddd{y_{n}-y_{n-1}}{h}+\dddd{h}{2}y''_n+O\llll( h^2 \rrrr),$$
we obtain the second-order expansion formula 
\begin{align}\label{11_10}
\mathcal{A}_n^{\bar{\ssss}}[y]=\dddd{1}{2\GGGG(1-\aaaa)h^\aaaa}\sum_{k=0}^{n}\bar{\ssss}_k^{(\aaaa)} y_{n-k}= y_n^{(\aaaa)}+\Bigg(&\dddd{\zzzz(\aaaa)-2\zzzz(\aaaa-1)}{2\GGGG(1-\aaaa)}\Bigg)y''_n h^{2-\aaaa}\nonumber\\
&+O\llll( h^{2} \rrrr),
\end{align}
where $\bar{\ssss}_0^{(\aaaa)}=1-2\zzzz(\aaaa),\quad \bar{\ssss}_1^{(\aaaa)}=\dddd{1}{2^\aaaa}+2\zzzz(\aaaa)$, $\bar{\ssss}_k^{(\aaaa)}=\oooo_k^{(\aaaa)},\; (k=2,\cdots,n)$.

The sum of the weights $\bar{\ssss}_{k}^{(\aaaa)}$ of approximation $\mathcal{A}_n^{\bar{\ssss}}[y]$  is zero and  
$$\mathcal{A}_n^{\bar{\ssss}}[1]=D^\aaaa 1 =0.$$
  Approximation $\mathcal{A}_n^{\bar{\ssss}}[y]$ has accuracy  $O\llll( h^{2-\aaaa} \rrrr)$ when  $y^\prime (0)=0$. Denote
$$W_n[\aaaa]=\sum _{k=1}^{n-1} \frac{1}{k^\aaaa}-\frac{n^{1-\aaaa}}{1-\aaaa}-\zeta (\aaaa).$$
\begin{clm} Let $y(x)=x$. Then
$$\mathcal{A}_n^{\bar{\ssss}}[y(x)]-y^{(\aaaa)}(x)=\dddd{W_n[\aaaa]h^{1-\aaaa}}{\GGGG(1-\aaaa)}.$$
\end{clm}
\begin{proof}
$$2\GGGG(1-\aaaa)h^\aaaa \mathcal{A}_n^{\bar{\ssss}}[x]=\sum_{k=0}^{n}\bar{\ssss}_k^{(\aaaa)}(x-k h) =-h\sum_{k=1}^{n}k\bar{\ssss}_k^{(\aaaa)},$$
\begin{align*}
\mathcal{A}_n^{\bar{\ssss}}[x]=-\dddd{h^{1-\aaaa}}{2\GGGG(1-\aaaa)}\Bigg( \dddd{1}{2^\aaaa}+2\zzzz(\aaaa)+\sum_{k=2}^{n-2}\Bigg( &\dddd{k}{(k+1)^\aaaa}-  \dddd{k}{(k-1)^\aaaa} \Bigg) \\
&-\dddd{n-1}{(n-2)^\aaaa} -\dddd{n}{(n-1)^\aaaa}\Bigg).
\end{align*}
By changing the index of summation we obtain
$$\mathcal{A}_n^{\bar{\ssss}}[x]=\dddd{h^{1-\aaaa}}{\GGGG(1-\aaaa)} \left(\sum _{k=1}^{n-1} \dddd{1}{k^\aaaa}-\zeta (\aaaa)\right).$$
Then
\begin{align*}
\mathcal{A}_n^{\bar{\ssss}}[y(x)]-y^{(\aaaa)}(x)&=\dddd{h^{1-\aaaa}}{\GGGG(1-\aaaa)} \left(\sum _{k=1}^{n-1} \dddd{1}{k^\aaaa}-\zeta (\aaaa)\right)-\dddd{x^{1-\aaaa}}{\GGGG(2-\aaaa)}\\
&=\dddd{h^{1-\aaaa}}{\GGGG(1-\aaaa)} \left(\sum _{k=1}^{n-1} \frac{1}{k^\aaaa}-\zeta (\aaaa)-\frac{n^{1-\aaaa}}{1-\aaaa}\right)=\dddd{h^{1-\aaaa}W_n[\aaaa]}{\GGGG(1-\aaaa)}.
\end{align*}
\end{proof}
Now we derive an approximation for the Caputo derivative of order $2-\aaaa$ by modifying the weights of approximation $\mathcal{A}_n^{\bar{\ssss}}[y]$ with index $n-1$ and $n$.
$$y(x)=y(x)-y(0)-y'(0)x+y(0)+y'(0)x=z(x)+y(0)+y'(0)x.$$
The function $z(x)=y(x)-y(0)-y'(0)x$ satisfies $z(0)=z'(0)=0$. Then 
$$\mathcal{A}^{\bar{\ssss}}_n[z]=z^{(\aaaa)}(x)+O\llll( h^{2-\aaaa} \rrrr)=y^{(\aaaa)}(x)-y'(0)\dddd{x^{1-\aaaa}}{\GGGG(2-\aaaa)}+O\llll( h^{2-\aaaa} \rrrr).$$
We use Claim 3 to determine the first term of the expansion of approximation $\mathcal{A}_n^{\bar{\ssss}}[y]$ at the initial point $x=0$.
$$\mathcal{A}^{\bar{\ssss}}_n[y]=\mathcal{A}^{\bar{\ssss}}_n[z]+\mathcal{A}^{\bar{\ssss}}_n[y(0)+y'(0)x],$$
$$\mathcal{A}^{\bar{\ssss}}_n[y]=y^{(\aaaa)}(x)-y'(0)\dddd{x^{1-\aaaa}}{\GGGG(2-\aaaa)}+y'(0)\mathcal{A}_n[x]+O\llll( h^{2-\aaaa} \rrrr),$$
\begin{align*}
\dddd{1}{2\GGGG(1-\aaaa)h^\aaaa}\sum_{k=0}^{n}\bar{\ssss}_k y_{n-k}=y_n^{(\aaaa)}+\llll(\dddd{y'_0 W_n[\aaaa]}{\GGGG(1-\aaaa)}\rrrr)h^{1-\aaaa}+O\llll( h^{2-\aaaa} \rrrr).
\end{align*}
By approximating $y'_0$ using  forward difference
$y'_0=(y_1-y_{0})/h+O(h)$ we obtain  approximation \eqref{5_30} for the Caputo derivative
\begin{align*}
\mathcal{A}_n^{\ssss}[y]=\dddd{1}{2\GGGG(1-\aaaa)h^\aaaa}\sum_{k=0}^{n}\ssss_k^{(\aaaa)} y_{n-k}= y_n^{(\aaaa)}+O\llll( h^{2-\aaaa} \rrrr),
\end{align*}
where $\ssss_k^{(\aaaa)}=\bar{\ssss}_k^{(\aaaa)}$ for $k=0,\cdots,n-2$ and
$$\ssss_{n-1}^{(\aaaa)}=-\dddd{1}{(n-2)^\aaaa}-2 W_n[\aaaa],\; \ssss_{n}^{(\aaaa)}=-\dddd{1}{(n-1)^\aaaa}+2 W_n[\aaaa].$$
When $n=2$, approximation $\mathcal{A}_n^{\ssss}[y]$ has weights
$$\ssss_{0}^{(\aaaa)}=1-2 \zeta (\aaaa),\; \ssss_{1}^{(\aaaa)}=4 \zeta (\aaaa)+\frac{2^{2-\aaaa}}{1-\aaaa}-2,\; \ssss_{2}^{(\aaaa)}=-2 \zeta (\aaaa)-\frac{2^{2-\aaaa}}{1-\aaaa}+1.$$
\setlength{\tabcolsep}{0.5em}
{ \renewcommand{\arraystretch}{1.1}
\begin{table}[ht]
	\caption{Error and order of numerical solution $NS[9]$ of Equation I with $\aaaa=0.25$, Equation II with $\aaaa=0.5$ and Equation III with $\aaaa=0.75$.}
	\small
	\centering
  \begin{tabular}{ l | c  c | c  c | c  c }
		\hline
		\hline
		\multirow{2}*{ $\quad \boldsymbol{h}$}  & \multicolumn{2}{c|}{{\bf Equation I}} & \multicolumn{2}{c|}{{\bf Equation II}}  & \multicolumn{2}{c}{{\bf Equation III}} \\
		\cline{2-7}  
   & $Error$ & $Order$  & $Error$ & $Order$  & $Error$ & $Order$ \\ 
		\hline \hline
$0.003125$    & $0.0000673$          & $1.6772$  & $0.0000637$          & $1.4780$    & $0.0026324$   & $1.2410$       \\ 
$0.0015625$   & $0.0000208$          & $1.6914$  & $0.0000227$          & $1.4846$    & $0.0011108$   & $1.2448$       \\ 
$0.00078125$  & $6.4\times 10^{-6}$  & $1.7025$  & $8.1\times 10^{-6}$  & $1.4892$    & $0.0004680$   & $1.2470$        \\ 
$0.000390625$ & $2.0\times 10^{-6}$  & $1.7112$  & $2.9\times 10^{-6}$  & $1.4924$    & $0.0001970$   & $1.2483$        \\
\hline
  \end{tabular}
	\end{table}
	}
	
The accuracy of the $L1$ approximation \eqref{3_10} and approximation \eqref{5_30} is $O\llll(h^{2-\aaaa}\rrrr)$. Now we show that the weights of approximation \eqref{5_30} satisfy \eqref{3_20}.	The Riemann zeta function is negative and decreasing on the interval $[0,1]$. The weight $\ssss_0^{(\alpha)}=1-\zzzz(\aaaa)$ of approximation \eqref{5_30} is positive and  $\ssss_1^{(\alpha)}=\dddd{1}{2^\aaaa}+\zzzz(\aaaa)$ is decreasing as a function of $\aaaa$. Then $\ssss_1^{(\alpha)}<\ssss_1^{(0)}=0$.
	From the binomial formula, 
	$$\ssss_k^{(\alpha)}=\llll(\dddd{1}{k^\aaaa}-\dddd{\aaaa}{k^{1+\aaaa}}\rrrr)-\llll(\dddd{1}{k^\aaaa}+\dddd{\aaaa}{k^{1+\aaaa}}\rrrr)+O\llll(\dddd{1}{k^{3+\aaaa}}\rrrr)=\dddd{2\aaaa}{k^{1+\aaaa}}+O\llll(\dddd{1}{k^{3+\aaaa}}\rrrr)$$
The number $\ssss_k^{(\alpha)}\approx \dddd{2\aaaa}{k^{1+\aaaa}}$ when  $k$ is a sufficiently large positive integer. In Claim 4 and Claim 5 we obtain estimates for the weights $\ssss_{n-1}^{(\alpha)}$ and $\ssss_n^{(\alpha)}$
The Bernoulli inequality holds when $r<0$ and $x>-1$.
$$(1+x)^r>1+r x.$$ 
					\begin{clm}
$$\ssss_{n-1}^{(\aaaa)}<-\dddd{11\aaaa}{6n^{1+\aaaa}}.$$
			\end{clm}
			\begin{proof}  From the Bernoulli inequality 
			$$\dddd{1}{(n-2)^\aaaa}=n^{-\aaaa}\llll(1-\dddd{2}{n}\rrrr)^{-\aaaa}>n^{-\aaaa}\llll(1+\dddd{2\aaaa}{n}\rrrr)=\dddd{1}{n^\aaaa}+\dddd{2\aaaa}{n^{1+\aaaa}}.$$
			From the formula for the sum of powers
\begin{align}\label{14_05}
\sum _{k=1}^{n-1} \frac{1}{k^\aaaa}=\zeta (\aaaa)+\frac{n^{1-\aaaa}}{1-\aaaa}\sum_{m=0}^\infty \binom{1-\aaaa}{m}\dddd{B_m}{n^m},
\end{align}
we obtain the following estimates for $S_n(\aaaa)$
	\begin{align}\label{14_10}
	\dddd{n^{1-\aaaa}}{1-\aaaa}-\dddd{1}{2 n^\aaaa}-\dddd{\aaaa}{12 n^{1+\aaaa}}<\sum _{k=1}^{n-1} \frac{1}{k^\aaaa}-\zzzz(\aaaa)<\dddd{n^{1-\aaaa}}{1-\aaaa}-\dddd{1}{2 n^\aaaa}.
	\end{align}
	Then
	$$\ssss_{n-1}^{(\aaaa)}=-\dddd{1}{(n-2)^\aaaa}-2 \left(S_n[\aaaa]-\frac{n^{1-\aaaa}}{1-\aaaa}\right)<-\dddd{1}{(n-2)^\aaaa}+\dddd{1}{n^\aaaa}+\dddd{\aaaa}{6n^{1+\aaaa}}.$$
	From \eqref{14_10}
		$$\ssss_{n-1}^{(\aaaa)}<-\dddd{1}{n^\aaaa}-\dddd{2\aaaa}{n^{1+\aaaa}}+\dddd{1}{n^\aaaa}+\dddd{\aaaa}{6n^{1+\aaaa}}=-\dddd{11\aaaa}{6n^{1+\aaaa}}.$$
			\end{proof}
				\begin{clm}
$$-\dddd{2}{(n-1)^\aaaa}<\ssss_{n}^{(\aaaa)}<-\dddd{2}{n^\aaaa}.$$
			\end{clm}
			\begin{proof} From \eqref{14_10}
			$$\ssss_{n}^{(\aaaa)}=-\dddd{1}{(n-1)^\aaaa}+2 \left(S_n[\aaaa]-\frac{n^{1-\aaaa}}{1-\aaaa}\right)<-\dddd{1}{(n-1)^\aaaa}-\dddd{1}{n^\aaaa}<-\dddd{2}{n^\aaaa}.$$
From the Bernoulli inequality
\begin{align}\label{15_10}
	\dddd{1}{(n-1)^\aaaa}=n^{-\aaaa}\llll(1-\dddd{1}{n}\rrrr)^{-\aaaa}>n^{-\aaaa}\llll(1+\dddd{\aaaa}{n}\rrrr)=\dddd{1}{n^\aaaa}+\dddd{\aaaa}{n^{1+\aaaa}}.
	\end{align}
From \eqref{14_10} and \eqref{15_10}
				$$\ssss_{n}^{(\aaaa)}>-\dddd{1}{(n-1)^\aaaa}-\dddd{1}{n^\aaaa}-\dddd{\aaaa}{6 n^{1+\aaaa}}>-\dddd{2}{(n-1)^\aaaa}.$$
			\end{proof}
	\section{ Second-order approximation for the Caputo derivative}
	We use approximation \eqref{11_10} to derive a second-order approximation for the Caputo derivative. By approximating $y''_n$ in \eqref{11_10} using a second-order backward difference approximation
$y''_n=(y_n-2y_{n-1}+y_{n-2})/h^2+O\llll( h \rrrr),$
we obtain the  approximation for the Caputo derivative
\begin{align}
\dddd{1}{2\GGGG(1-\aaaa)h^\aaaa}\sum_{k=0}^{n}\bar{\dddddd}_k^{(\aaaa)} y_{n-k}= y_n^{(\aaaa)}+O\llll(h^2\rrrr),
\end{align}\label{15_20}
where $\bar{\dddddd}_k^{(\aaaa)}=\bar{\ssss}_k^{(\aaaa)}$ for $k=3,\cdots,n$ and
$$\bar{\dddddd}_0^{(\aaaa)}=\bar{\ssss}_0^{(\aaaa)}-\zzzz(\aaaa)+2\zzzz(\aaaa-1)=1-3 \zzzz(\aaaa)+2 \zzzz(\aaaa-1),$$
$$\bar{\dddddd}_1^{(\aaaa)}=\bar{\ssss}_1^{(\aaaa)}+2\zzzz(\aaaa)-4\zzzz(\aaaa-1)=\dddd{1}{2^\aaaa}+4\zzzz(\aaaa)-4 \zzzz(\aaaa-1),$$
$$\bar{\dddddd}_2^{(\aaaa)}=\bar{\ssss}_2^{(\aaaa)}-\zzzz(\aaaa)+2\zzzz(\aaaa-1)=\dddd{1}{3^\aaaa}-1-\zzzz(\aaaa)+2\zzzz(\aaaa-1).$$
Approximation \eqref{19_20} has second-order accuracy  when $y'(0)=0$.
By using the method from the previous section we obtain the second-order approximation
\begin{align*}
\mathcal{A}_n^\dddddd[y]=\dddd{1}{2\GGGG(1-\aaaa)h^\aaaa}\sum_{k=0}^{n}\dddddd_k^{(\aaaa)} y_{n-k}=y_n^{(\aaaa)}+O\llll( h^2 \rrrr),
\end{align*}
where $\dddddd_k^{(\aaaa)}=\bar{\dddddd}_k^{(\aaaa)}$ for $k=0,\cdots,n-2$ and
\begin{align*}
\dddddd_{n-1}^{(\aaaa)}=-\dddd{1}{(n-2)^\aaaa}-2 W_n[\aaaa],\;\dddddd_{n}^{(\aaaa)}=-\dddd{1}{(n-1)^\aaaa}+2 W_n[\aaaa].
\end{align*}
Approximation $\mathcal{A}_n^\dddddd[y]$ has a second-order accuracy $O\llll(h^2\rrrr)$ and accuracy $O\llll( h^{2-\aaaa} \rrrr)$ when $n$ is a {\it small} positive number. When $n=2$, approximation $\mathcal{A}_n^\dddddd[y]$ has weights
$\dddddd_{0}^{(\aaaa)}=1+2 \zeta (\aaaa-1)-3 \zeta (\aaaa)$ and
$$\dddddd_{1}^{(\aaaa)}=-4 \zeta (\aaaa-1)+6 \zeta (\aaaa)+\frac{2^{2-\aaaa}}{1-\aaaa}-2,\; \dddddd_{2}^{(\aaaa)}=1+2 \zeta (\aaaa-1)-3 \zeta (\aaaa)-\frac{2^{2-\aaaa}}{1-\aaaa}.$$
When $n=3$: $\dddddd_{0}^{(\aaaa)}=1+2 \zeta (\aaaa-1)-3 \zeta (\aaaa),\;\dddddd_{1}^{(\aaaa)}=\dddd{1}{2^\aaaa}-4 \zeta (\aaaa-1)+4 \zeta (\aaaa),$
$$\dddddd_{2}^{(\aaaa)}=2 \zeta (\aaaa-1)+\zeta (\aaaa)-2\llll(\dddd{1}{2^\aaaa}-\dddd{ 3^{1-\aaaa}}{1-\aaaa}\rrrr)-3,\;\dddddd_{3}^{(\aaaa)}=2-2 \zeta (\aaaa)-\dddd{2\times 3^{1-\aaaa}}{1-\aaaa}+\dddd{1}{2^\aaaa}.$$
In Table 4 we present the numerical results for the numerical solutions of  Equation I, Equation II and Equation III using approximation $\mathcal{A}_n^\dddddd[y]$  for the Caputo derivative. In Figure 2 we compare the numerical solutions of Equation III using approximations \eqref{10_20}, \eqref{5_30} and $\mathcal{A}_n^\dddddd[y]$  when $\aaaa=0.6$.
\setlength{\tabcolsep}{0.5em}
{ \renewcommand{\arraystretch}{1.1}
\begin{table}[ht]
	\caption{Error and order of numerical solution $NS[10]$ of Equation I and $\aaaa=0.25$, Equation II and $\aaaa=0.5$ and Equation III,$\aaaa=0.75$.}
	\small
	\centering
  \begin{tabular}{ l | c  c | c  c | c  c }
		\hline
		\hline
		\multirow{2}*{ $\quad \boldsymbol{h}$}  & \multicolumn{2}{c|}{{\bf Equation I}} & \multicolumn{2}{c|}{{\bf Equation II}}  & \multicolumn{2}{c}{{\bf Equation III}} \\
		\cline{2-7}  
   & $Error$ & $Order$  & $Error$ & $Order$  & $Error$ & $Order$ \\ 
		\hline \hline 
$0.003125$    & $0.0000171$         & $1.9751$  & $2.1\times 10^{-6}$   & $1.9293$    & $0.0001178$          & $1.9776$       \\ 
$0.0015625$   & $4.3\times 10^{-6}$  & $1.9854$  & $5.4\times 10^{-7}$   & $1.9518$    & $0.0000297$          & $1.9871$       \\ 
$0.00078125$  & $1.1\times 10^{-6}$  & $1.9914$  & $1.4\times 10^{-7}$   & $1.9668$    & $7.5\times 10^{-6}$  & $1.9924$        \\ 
$0.000390625$ & $2.8\times 10^{-7}$  & $1.9949$  & $3.5\times 10^{-8}$   & $1.9770$    & $1.9\times 10^{-6}$  & $1.9955$        \\
\hline
  \end{tabular}
	\end{table}
	}
			\begin{figure}[ht]
  \centering
  \caption{Graph of the exact solution of Equation III  and  numerical solutions $NS[9]$-red,  and $NS[10]$-blue, and $NS[20]$-green for $h=0.1$ and $\alpha=0.6$.}
  \includegraphics[width=0.55\textwidth]{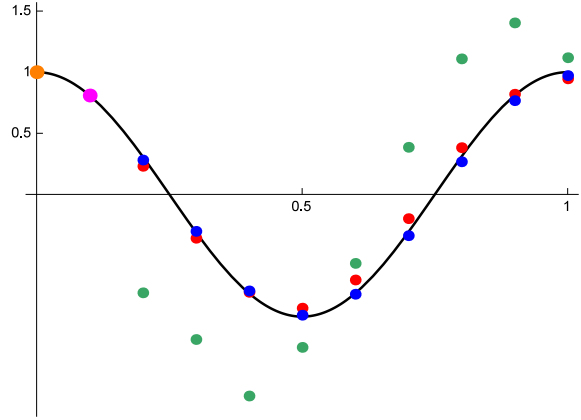}
\end{figure}

Approximations \eqref{4_20} and \eqref{6_05}  have  second-order accuracy and their weights satisfy \eqref{4_30}.  In Lemma 7 and Lemma 8 we prove that $\dddddd_0^{(\alpha)}>0$ and $\dddddd_1^{(\alpha)}<0$. In the proof of Lemma 7 and Lemma 8 we use the estimates from Claim 6 and the  Laurent expansion of the Riemann zeta function
$$\zzzz(\aaaa)=\dddd{1}{\aaaa-1}+\sum_{n=0}^{\infty}(-1)^n \dddd{\gggg_n}{n!}(\aaaa-1)^n,$$
where $\gggg_n$ are the the Stieltjes constants 
$$\gggg_n=\lim_{m\rightarrow \infty}\llll(\sum_{k=1}^{m}\dddd{\ln^n k}{k}-\dddd{\ln^{n+1}m}{n+1}\rrrr).$$
The Stieltjes constant $\gggg_0= 0.57721$ is equal to the Euler-Mascheroni constant and $\gggg_1= -0.0728<0,\; \gggg_2= -0.00969<0,\; \gggg_3= 0.002>0 $. 
\begin{clm} Let $0<\aaaa<1$. Then
$$\dddd{1}{1-\aaaa}>1+\aaaa+\aaaa^2,\quad \dddd{2}{2-\aaaa}<1+\aaaa,$$
$$\dddd{1}{2^\aaaa}<1-\aaaa\ln 2+\dddd{\ln 2}{2}\aaaa^2
,\quad \dddd{1}{3^\aaaa}>1-\aaaa\ln 3.$$
\end{clm}
The estimates of Claim 6 are obtained from the Maclaurin series  expansion of the functions $1/(1-\aaaa),2^{-\aaaa}=e^{-\aaaa \ln 2}$ and $3^{-\aaaa}=e^{-\aaaa \ln 3}$.
\begin{lem} 
$$\dddddd_1^{(\aaaa)}=\dddd{1}{2^\aaaa}+4\zzzz(\aaaa)-4 \zzzz(\aaaa-1)<0.$$
\end{lem}
\begin{proof} From the Laurent expansion of the Riemann zeta function 
\begin{align}\label{18_10}
\zzzz(\aaaa)<\dddd{1}{\aaaa-1}+\gggg_0-\gggg_1 (\aaaa-1),
\end{align}
\begin{align}\label{18_20}
\zzzz(\aaaa-1)>\dddd{1}{\aaaa-2}+\gggg_0-\gggg_1 (\aaaa-2)+\dddd{\gggg_2}{2} (\aaaa-2)^2.
\end{align}
Then
$$\dddddd_1^{(\aaaa)}<\dddd{1}{2^\aaaa}+\dddd{4}{\aaaa-1}-\dddd{4}{\aaaa-2}-4\gggg_1 -2\gggg_2 (\aaaa-2)^2<0.$$
We have that
$$-4\gggg_1 -2\gggg_2 (\aaaa-2)^2<-4\gggg_1-8\gggg_2<0.37.$$
Then
$$\dddddd_1^{(\aaaa)}<\dddd{1}{2^\aaaa}-\dddd{4}{1-\aaaa}+\dddd{4}{2-\aaaa}+0.37.$$
From Claim 6
$$\dddddd_1^{(\aaaa)}<1-\aaaa\ln 2+\dddd{\ln 2}{2}\aaaa^2-4(1+\aaaa+\aaaa^2)+2(1+\aaaa)+0.37,$$
$$\dddddd_1^{(\aaaa)}<-\dddd{1}{2}\llll(8-\ln 2\rrrr)\aaaa^2-(2+\ln 2)\aaaa-0.63. $$
The function $-\dddd{1}{2}\llll(8-\ln 2\rrrr)\aaaa^2-(2+\ln 2)\aaaa$ is decreasing on $[0,1]$ . Then
$$\dddddd_1^{(\aaaa)}<\dddddd_1^{(0)}<-0.63<0.$$
\end{proof}
\begin{lem} 
$$\dddddd_2^{(\aaaa)}=\dddd{1}{3^\aaaa}-1-\zzzz(\aaaa)+2\zzzz(\aaaa-1)>0.$$
\end{lem}
\begin{proof} From \eqref{18_10} and \eqref{18_20}
$$\dddddd_2^{(\aaaa)}>\dddd{1}{3^\aaaa}-1+\dddd{1}{1-\aaaa}-\dddd{2}{2-\aaaa}+\gggg_0+\gggg_1 (3-\aaaa)+\gggg_2 (\aaaa-2)^2.$$
The numbers $\gggg_1$ and $\gggg_2$ are negative. Then
$$\gggg_0+\gggg_1 (3-\aaaa)+\gggg_2 (\aaaa-2)^2>\gggg_0+3\gggg_1+4\gggg_2>0.32,$$
$$\dddddd_2^{(\aaaa)}>\dddd{1}{3^\aaaa}-1+\dddd{1}{1-\aaaa}-\dddd{2}{2-\aaaa}+0.32.$$
From Claim 6
$$\dddddd_2^{(\aaaa)}>1-\aaaa \ln 3-1+1+\aaaa+\aaaa^2-(1+\aaaa)+0.32,$$
$$\dddddd_2^{(\aaaa)}>\aaaa^2-\aaaa \ln 3+0.32.$$
The function $\aaaa^2-\aaaa \ln 3$ has a minimum at the point $\aaaa=\ln 3/2$. Then
$$\dddddd_2^{(\aaaa)}>0.32-\dddd{\ln^2 3}{4}>0.018>0.$$
\end{proof}
	\section{Numerical Solution of the Fractional Relaxation Equation}
The fractional relaxation equation  is a linear ordinary fractional differential equation of order $\aaaa$, where $0<\aaaa<1$.
	\begin{align}\label{19_10}
	y^{(\aaaa)}(x)+D y(x)=F(x),\quad y(0)=y_0.
	\end{align}
	The analytical and numerical solutions of the fractional relaxation equation are discussed in \cite{Diethelm2010,Dimitrov2014,Dimitrov2016,GulsuOzturkAnapali2013,Podlubny1999,Mainardi1996}.
	In this section we derive the numerical solution of the fractional relaxation equation \eqref{19_10}. In Theorem 9 and Theorem 10 we prove that the numerical solution of equation \eqref{19_10} on the interval $[0,X]$ which uses  approximations \eqref{3_10}, \eqref{5_30} and \eqref{6_20} for the Caputo derivative converges to the exact solution when $D>D_0$, where $D_0$ is a negative number. In \cite{Dimitrov2015} we showed that 
\begin{align}\label{19_20}
\bar{y}_1=\dddd{y(0)+\Gamma(2-\aaaa)h^\aaaa F(h)}{1+\Gamma(2-\aaaa)D h^\aaaa}
\end{align}
is an approximation for the value of the exact solution $y(h)$ of equation \eqref{19_10} with accuracy $O\llll(h^2\rrrr)$. Let $h=X/n$ and 
\begin{align}\label{20_10}\tag{*}
	\dddd{1}{h^\aaaa}\llll(\llllll_0y_m-\sum_{k=1}^{m}\llllll_k y_{m-k} \rrrr)\approx y_m^{(\aaaa)},
	\end{align}
be an approximation for the Caputo derivative.
By approximating the Caputo derivative in equation \eqref{19_10} at the point $x_m=m h$ we obtain 
	$$\dddd{1}{h^\aaaa}\llll(\llllll_0 y_m-\sum_{k=1}^{m}\llllll_k y_{m-k} \rrrr)+Dy_m\approx F_m,$$
$$\llll(\llllll_0+Dh^\aaaa\rrrr) y_m-\sum_{k=1}^{m}\llllll_k y_{m-k}\approx h^\aaaa F_m.$$
	The numerical solution $\{u_m\}_{m=0}^n$ of equation \eqref{19_10}, where $u_m\approx y_m$ is computed with $u_0=y_0, u_1=\bar{y}_1$ and
\begin{align}\tag{$\boldsymbol{NS[*]}$}
		u_m=\dddd{1}{\llllll_0+Dh^\aaaa}\llll(h^\aaaa F_m+\sum_{k=1}^{m}\llllll_k u_{m-k}\rrrr).
		\end{align}
Suppose that approximation \eqref{20_10} has accuracy $O\llll(h^{2-\aaaa}\rrrr)$, where the weights  $\llllll_k$ are positive for all $k=0,1,\cdots,m$ and satisfy
\begin{align}\label{20_20}
	\sum_{k=1}^{m}\llllll_k=\llllll_0,\quad 0<\dddd{L}{m^\aaaa}<\llllll_m.
	\end{align}
The weights of approximations \eqref{3_10}, \eqref{5_30} and \eqref{6_20} satisfy \eqref{20_20}.	Let  $E_m h^{2-\aaaa}$ be the error of approximation \eqref{20_10} at the point $x_m=mh$, and $E$ be a positive number such that $|E_m|<E$ for all $m=2,3,\cdots,n$. The error  $e_m=u_m-y_m$ of numerical solution $NS(*)$ satisfies $e_0=0, e_1=\bar{y}_1-y(h)$ and
			$$e_m=\dddd{1}{\llllll_0+D h^\aaaa}\sum_{k=1}^{m}\llllll_k e_{m-k}+A_m h^2,$$
where $A_m=\dddd{E_m}{\llllll_0+Dh^\aaaa}$. Let $A>2E/\llllll_0$ be a positive number such that $|e_1|<A h^{2-\aaaa}$. When $D>0$ the numbers $A_m$ satisfy 	 
			$$|A_m|<\dddd{|E_m|}{\llllll_0}<\dddd{E}{\llllll_0}<A.$$
When $D<0$ 	and $|D| h^\aaaa<\llllll_0/2$  
	$$|A_m|=\dddd{|E_m|}{\llllll_0+Dh^\aaaa}<\dddd{2|E_m| }{\llllll_0}<\dddd{2E}{\llllll_0}<A.$$
			\begin{thm} 	Let $D>0$ and $C>Max\{A,A(\llllll_0+D)/D\}$. The error of numerical solution $NS[*]$ of equation \eqref{19_10} on the interval $[0,X]$ satisfies
			$$|e_m|<Ch^{2-\aaaa}.$$
			\end{thm}
			\begin{proof} Induction on $m$.
			Suppose that $|e_k|<Ch^{2-\aaaa}$ for $k=1,\cdots,m-1$. The error $e_m$ satisfies
$$|e_m|\leq\dddd{1}{\llllll_0+Dh^\aaaa}\sum_{k=1}^{m-1}\llllll_k |e_{m-k}|+|A_m| h^2.$$
From the induction assumption
$$|e_m|<\dddd{Ch^{2-\aaaa}}{\llllll_0+D h^\aaaa}\sum_{k=1}^{m-1}\llllll_k +A h^2,$$
$$|e_m|<\dddd{C \llllll_0h^{2-\aaaa}}{\llllll_0+Dh^\aaaa} +A h^2=\llll( \dddd{C \llllll_0+A\llll(\llllll_0+D h^\aaaa\rrrr)h^{\aaaa}}{\llllll_0+Dh^\aaaa} \rrrr)h^{2-\aaaa},$$
$$|e_m|<\llll( \dddd{\llllll_0+A\llll(\llllll_0+D\rrrr)h^{\aaaa}/C}{\llllll_0+Dh^\aaaa} \rrrr)C h^{2-\aaaa}.$$
We have that $A(\llllll_0+D)/C<D$, because $C>A(\llllll_0+D)/D$. Hence
$$|e_m|<C h^{2-\aaaa}.$$
	\end{proof}
			\begin{thm} Let $- L/X^\aaaa<D<0$. The error of numerical solution $NS[*]$ of equation \eqref{19_10} on the interval $[0,X]$ satisfies 
			$$|e_m|<Ch^{2-\aaaa}$$
			where $C=Max\llll\{A,\dddd{A X^\aaaa (\llllll_0+D)}{L-D X^\aaaa}\rrrr\}$.
			\end{thm}
			\begin{proof} Induction on $m$. 
			Suppose that $|e_k|<Ch^{2-\aaaa}$ for $k=1,\cdots,m-1$. The error $e_m$ satisfies
$$|e_m|\leq\dddd{1}{\llllll_0+Dh^\aaaa}\sum_{k=1}^{m-1}\llllll_k |e_{m-k}|+|A_m| h^2,$$
$$|e_m|<\dddd{Ch^{2-\aaaa}}{\llllll_0+Dh^\aaaa}\sum_{k=1}^{m-1}\llllll_k +A h^2,$$
$$|e_m|<\dddd{C (\llllll_0-\llllll_m)h^{2-\aaaa}}{\llllll_0+Dh^\aaaa} +A h^2=\llll( \dddd{ \llllll_0- \llllll_m+A\llll(\llllll_0+D h^\aaaa \rrrr)h^{\aaaa}/C}{\llllll_0+Dh^\aaaa} \rrrr)Ch^{2-\aaaa}.$$
From \eqref{20_20}
$$\llllll_m>\dddd{L}{m^\aaaa}=\dddd{L}{X^\aaaa}\llll(\dddd{X}{m} \rrrr)^\aaaa>\dddd{L}{X^\aaaa}\llll(\dddd{X}{n} \rrrr)^\aaaa=\dddd{L}{X^\aaaa}h^\aaaa.$$
Then
$$\llllll_m - A\llll(\llllll_0+D h^\aaaa \rrrr)h^{\aaaa}/C>\llll(\dddd{L}{X^\aaaa}- \dddd{A\llll(\llllll_0+D \rrrr)}{C}\rrrr)h^\aaaa.$$
We have that
$$\dddd{L}{X^\aaaa}- \dddd{A\llll(\llllll_0+D \rrrr)}{C}>D,$$
because 
$$C>\dddd{A X^\aaaa (\llllll_0+D)}{L-D X^\aaaa}.$$
\end{proof}
{\bf Example:}	The fractional relaxation equation				
\begin{align} \label{22_10}
y^{(\aaaa)}(x)+Dy(x)=x^{1-\aaaa}E_{1,2-\aaaa}(x)+De^x,\quad y(0)=1,
\end{align}
has the solution $y(x)=e^x$. In Table 5 we compute the maximum error and the order of the numerical solutions of equation \eqref{22_10} which use approximations \eqref{3_10}, \eqref{5_30} and \eqref{6_20} for the Caputo derivative when $D=-1$ and $\aaaa=0.6$. In Table 6 we compute the error and the order of the numerical solution of equation \eqref{22_10} for the $L1$ approximation \eqref{3_10} and values of $D=-2,-5,-7$ and $\aaaa=0.5$. 
 The numerical results from Table 6 suggest that  when $D<-5$, numerical solutions  $NS[1]$, $NS[9]$ and $NS[12]$ of equation \eqref{22_10} are not suitable for practical use.
\setlength{\tabcolsep}{0.5em}
{ \renewcommand{\arraystretch}{1.1}
\begin{table}[ht]
	\caption{Maximum error and order of numerical solutions $NS[1]$, $NS[9]$ and $NS[12]$  of equation \eqref{22_10} with $D=-1,\aaaa=0.6$  on the interval $[0,1]$.}
	\small
	\centering
  \begin{tabular}{ l | c  c | c  c | c  c }
		\hline
		\hline
		\multirow{2}*{ $\quad \boldsymbol{h}$}  & \multicolumn{2}{c|}{{$\boldsymbol{NS[1]}$}} & \multicolumn{2}{c|}{$\boldsymbol{NS[9]}$}  & \multicolumn{2}{c}{$\boldsymbol{NS[12]}$} \\
		\cline{2-7}  
   & $Error$ & $Order$  & $Error$ & $Order$  & $Error$ & $Order$ \\ 
		\hline \hline
$0.003125$    & $0.0004594$  & $1.3890$  & $0.0005372$  & $1.3826$    & $0.0003854$   & $1.3980$       \\ 
$0.0015625$   & $0.0001750$  & $1.3927$  & $0.0002052$  & $1.3883$    & $0.0001462$   & $1.3989$       \\ 
$0.00078125$  & $0.0000665$  & $1.3952$  & $0.0000782$  & $1.3922$    & $0.0000554$   & $1.3994$        \\ 
$0.000390625$ & $0.0000253$  & $1.3968$  & $0.0000297$  & $1.3949$    & $0.0000210$   & $1.3997$        \\
\hline
  \end{tabular}
	\end{table}
	}
\setlength{\tabcolsep}{0.5em}
{ \renewcommand{\arraystretch}{1.1}
\begin{table}[ht]
	\caption{Maximum error and order of numerical solution $NS[1]$,  of equation \eqref{22_10} with $\aaaa=0.5$  on the interval $[0,1]$ and $D=-2,-5,-7$.}
	\small
	\centering
  \begin{tabular}{ l | c  c | c  c | c  c }
		\hline
		\hline
		\multirow{2}*{ $\quad \boldsymbol{h}$}  & \multicolumn{2}{c|}{$\boldsymbol{D=-2}$} & \multicolumn{2}{c|}{$\boldsymbol{D=-5}$}  & \multicolumn{2}{c}{$\boldsymbol{D=-7}$} \\
		\cline{2-7}  
   & $Error$ & $Order$  & $Error$ & $Order$  & $Error$ & $Order$ \\ 
		\hline \hline 
$0.003125$    & $0.00275681$  & $1.4795$  & $1.4\times 10^{6}$ & $2.0272$    & $6.7\times 10^{16}$   & $4.6422$       \\ 
$0.0015625$   & $0.00098584$  & $1.4836$  & $438076.6$         & $1.6668$    & $1.3\times 10^{16}$   & $2.5630$       \\ 
$0.00078125$  & $0.00035150$  & $1.4878$  & $150198.5$         & $1.5443$    & $3.1\times 10^{15}$   & $1.8610$        \\ 
$0.000390625$ & $0.00012503$  & $1.4913$  & $52927.21$         & $1.5048$    & $1.0\times 10^{15}$   & $1.6150$        \\
\hline
  \end{tabular}
	\end{table}
	}
\section{Fourier Transform Method for Constructing  Approximations of the Caputo Derivative}
The Fourier transform method is an important method for constructing approximations for fractional derivatives. The generating function of an approximation is directly related to the Fourier transform of the approximation. The Fourier transform method is used by Ding and Li \cite{DingLi2015,DingLi2016}, Lubich \cite{Lubich1986}, Tian et al. \cite{TianZhouDeng2012} for constructing approximations for the fractional derivative. In \cite{Dimitrov2016} we use the Fourier transform method and the series expansion formula of the polylogarithm function  to derive the asymptotic expansion formula \eqref{5_10} for the Riemann sum approximation of the fractional integral at the right endpoint. The  expansion formula for the left endpoint of \eqref{5_10} is obtained from the Euler-Maclaurin formula for the definite integral. In this section we use the  Fourier transform method and integration by parts to obtain approximations for the Caputo derivative of order $1-\aaaa,2-\aaaa$ and $3-\aaaa$. 

The Fourier Transform of the function $y$ is defined as
$$\mathcal{F}[y(x)](w)=\hat{y}(w)=\int_{-\infty}^{\infty}e^{i w t}y(t)d t.
$$
The Fourier transform has properties 
$$\mathcal{F}[y(x-b)](w)=e^{i w b} \hat{y}(w),\quad \mathcal{F}[y(x)*z(x)](w)= \hat{y}(w) \hat{z}(w),$$
$$\mathcal{F}[D^\aaaa y(x)](w)=(-iw)^\aaaa \hat{y}(w),\quad \mathcal{F}[I^\aaaa y(x)](w)=(-iw)^{-\aaaa} \hat{y}(w).$$
\subsection{Expansion Formula}
Let $S_n[y]$ be the Riemann sum of the formal fractional integral $K^{-\aaaa}y(x)$.
$$S_n[y]=\dddd{1}{h^\aaaa}\sum_{k=1}^{n-1}\dddd{y(x-kh)}{k^{1+\aaaa}}.$$
By applying Fourier transform to $S_n[y]$ and letting $n\rightarrow\infty$
$$\mathcal{F}[S_\infty[y]](w)=\dddd{1}{h^\aaaa}\sum_{k=1}^{\infty}\dddd{e^{ikhw}}{k^{1+\aaaa}}=\dddd{1}{h^\aaaa}Li_{1+\aaaa}\llll(e^{ikhw} \rrrr).$$
From \eqref{4_10} with $s=1+\aaaa$ and $x=e^{ikhw}$ we obtain
\begin{align}\label{24_10}
&Li_{1+\aaaa}\llll(e^{i w h}\rrrr)=\GGGG(-\aaaa)(-i w h)^{\aaaa}+\sum_{m=0}^\infty \dddd{\zzzz(1+\aaaa-m)}{m!}(i w h)^m,\nonumber\\
\dddd{1}{h^\aaaa}&Li_{1+\aaaa}\llll(e^{i w h}\rrrr)=\GGGG(-\aaaa)(-i w)^{\aaaa}+\sum_{m=0}^\infty \dddd{(-1)^m}{m!}\zzzz(1+\aaaa-m)(-i w)^m h^{m-\aaaa}.
\end{align}
Denote by $\mathcal{R}^{(\aaaa)}[y]$ and $\mathcal{L}^{(\aaaa)}[y]$ the left and right asymptotic sums of
$$S_n[y]-\GGGG(-\aaaa)y^{(\aaaa)}(x).$$
By applying inverse Fourier transform to \eqref{24_10} we obtain the expansion formula for the right asymptotic sum.
\begin{align*}
\mathcal{R}^{(\aaaa)}[y]=\sum_{m=0}^\infty \dddd{\zzzz(1+\aaaa-m)}{m!}y^{(m)}(x)h^{m-\aaaa}.
\end{align*}
The result is summarized in Lemma 11.
	\begin{lem} Let $y^{(m)}(0)=0$ for $m=0,1,\cdots$. Then
	\begin{align*}
\dddd{1}{h^\aaaa}\sum_{k=1}^{n-1}\dddd{y(x-kh)}{k^{1+\aaaa}}=\GGGG(-\aaaa)y^{(\aaaa)}(x)+\sum_{m=0}^\infty\dddd{(-1)^m}{m!}\zzzz(1+\aaaa-m)y^{(m)}(x)h^{m-\aaaa}.
\end{align*}
	\end{lem}
	From Lemma 11, we obtain the approximation for the Caputo derivative
	\begin{align}\label{25_10}
\dddd{1}{\GGGG(-\aaaa)h^\aaaa}\llll(\sum_{k=1}^{n-1}\dddd{y_{n-k}}{k^{1+\aaaa}}-\zzzz(1+\aaaa)y_n \rrrr)=y_n^{(\aaaa)}&-\dddd{\zzzz(\aaaa)}{\GGGG(-\aaaa)}y'_n h^{1-\aaaa}\\
&+\dddd{\zzzz(\aaaa-1)}{2\GGGG(-\aaaa)}y''_n h^{2-\aaaa}+O\llll( h^{3-\aaaa} \rrrr).\nonumber
\end{align}
Approximation \eqref{25_10} has accuracy $O\llll( h^{3-\aaaa} \rrrr)$ when the function $y(x)$ satisfies $y(0)=y'(0)=y''(0)=0$. Let
$$\oooo_0^{(\aaaa)}=-\zzzz(1+\aaaa),\;\oooo_k^{(\aaaa)}=\dddd{1}{k^{1+\aaaa}},\quad (k=1,\cdots,n-1),$$
and
$$\oooo_n^{(\aaaa)}=-\sum_{k=0}^{n-1}\oooo_k^{(\aaaa)}=\zzzz(1+\aaaa)-\sum_{k=1}^{n-1}\dddd{1}{k^{1+\aaaa}}.$$
\begin{cor}
\begin{align}\label{25_20}
\mathcal{A}^{\oooo}_n[y]=\dddd{1}{\GGGG(-\aaaa)h^\aaaa}\sum_{k=1}^{n}\oooo_k^{(\aaaa)} y_{n-k}= y_n^{(\aaaa)}+O\llll( h^{1-\aaaa}\rrrr).
\end{align}
\end{cor}
Approximation \eqref{25_20} satisfies $\mathcal{A}^{\oooo}_n[1]=0$, because $\sum_{k=1}^{n}\oooo_k^{(\aaaa)}=0$.
\setlength{\tabcolsep}{0.5em}
{ \renewcommand{\arraystretch}{1.1}
\begin{table}[ht]
	\caption{Error and order of numerical solution $NS[34]$ of Equation I with $\aaaa=0.25$, Equation II with $\aaaa=0.5$ and Equation III with $\aaaa=0.75$.}
	\small
	\centering
  \begin{tabular}{ l | c  c | c  c | c  c }
		\hline
		\hline
		\multirow{2}*{ $\quad \boldsymbol{h}$}  & \multicolumn{2}{c|}{{\bf Equation I}} & \multicolumn{2}{c|}{{\bf Equation II}}  & \multicolumn{2}{c}{{\bf Equation III}} \\
		\cline{2-7}  
   & $Error$ & $Order$  & $Error$ & $Order$  & $Error$ & $Order$ \\ 
		\hline \hline 
$0.003125$    & $0.0098745$  & $0.7512$  & $0.0298944$   & $0.5069$    & $0.370920$  & $0.2924$       \\ 
$0.0015625$   & $0.0058683$  & $0.7508$  & $0.0298944$   & $0.5051$    & $0.304271$  & $0.2858$       \\ 
$0.00078125$  & $0.0034881$  & $0.7505$  & $0.0148565$   & $0.5037$    & $0.250614$  & $0.2799$        \\ 
$0.000390625$ & $0.0020736$  & $0.7503$  & $0.0104857$   & $0.5027$    & $0.207125$  & $0.2750$        \\
\hline
  \end{tabular}
	\end{table}
	}
	Denote by $\mathcal{L}[y]=\mathcal{L}^{(1)}[y]$ the left asymptotic sum of
	$$h\sum_{k=1}^{n-1}y(kh)-\int_0^x y(t) d t.$$
From the	Euler-Maclaurin formula 
$$h\sum_{k=1}^{n-1}y(kh)=\int_0^x y(t) d t-(y(0)+y(x))\dddd{h}{2}+\sum_{m=1}^\infty \dddd{B_{2m}h^{2m}}{(2m)!}\llll(y^{(2m-1)}(x)-y^{(2m-1)}(0)\rrrr)$$
we obtain
$$\mathcal{L}[y]=-\dddd{y(0) h}{2}-\sum_{m=1}^\infty \dddd{B_{2m}}{(2m)!}y^{(2m-1)}(0)h^{2m}.$$
Now we use the Euler-Maclaurin formula and integration
	by parts to derive the formula for $\mathcal{L}^{(\aaaa)}[y]$.
			\begin{align*}
\GGGG(1-\aaaa)y^{(\aaaa)}(x)=\int_0^x \dddd{y'(t)}{(x-t)^\aaaa}dt=\int_0^{x/2} \dddd{y'(t)}{(x-t)^\aaaa}dt+\int_{x/2}^x \dddd{y'(t)}{(x-t)^\aaaa}dt.
\end{align*}
	We have that
\begin{align*}
	\int_0^{x/2} \dddd{y'(t)}{(x-t)^\aaaa}dt&=\llll.(x-t)^{-\aaaa} y(t) \rrrr|_0^{x/2}-\int_0^{x/2} y(t) d\llll[(x-t)^{-\aaaa} \rrrr]\\
	&=\llll(\dddd{x}{2}\rrrr)^{-\aaaa} y\llll(\dddd{x}{2}\rrrr)-x^{-\aaaa} y(0)-\aaaa \int_0^{x/2} \dddd{y(t)}{(x-t)^{1+\aaaa}} dt.
\end{align*}
	The gamma function satisfies $\GGGG(1-\aaaa)=-\aaaa \GGGG(-\aaaa)$. Then
			\begin{align}\label{26_10}
\GGGG(-\aaaa)y^{(\aaaa)}(x)=\dddd{ y(0)}{\aaaa x^{\aaaa}}+ \int_0^{x/2} &\dddd{y(t)}{(x-t)^{1+\aaaa}} dt\\
&-\dddd{1}{\aaaa}\llll(\llll(\dddd{x}{2}\rrrr)^{-\aaaa} y\llll(\dddd{x}{2}\rrrr)+\int_{x/2}^x \dddd{y'(t)}{(x-t)^\aaaa}dt\rrrr).\nonumber
\end{align}
Let
$$z(t)=\dddd{y(t)}{(x-t)^{1+\aaaa}}=y(t)Z(t),$$
where $Z(t)=(x-t)^{-1-\aaaa}$. From \eqref{26_10} we have that 
$$\mathcal{L}^{(\aaaa)}[y]=\mathcal{L}[z]+y(0)x^{-\aaaa}/\aaaa .$$
The function $Z(t)$ has derivatives
$$Z^{(k)}(t)=\dddd{\GGGG(1+\aaaa+k)}{\GGGG(1+\aaaa)}\dddd{1}{(x-t)^{1+\aaaa+k}}.$$
From the Leibnitz's rule
$$z^{(2m-1)}(0)=\sum_{k=0}^{2m-1}\binom{2m-1}{k}y^{(2m-1-k)}(0)Z^{(k)}(0).$$
Then
$$\mathcal{L}[z]=-\dddd{z(0) h}{2}-\sum_{m=1}^\infty \dddd{B_{2m}}{(2m)!}z^{(2m-1)}(0)h^{2m},$$
\begin{align}\label{27_10}
\mathcal{L}[z]=&-\dddd{y(0) h}{2 x^{1+\aaaa}}\\
&-\sum_{m=1}^\infty \dddd{B_{2m}}{(2m)!}\llll( \sum_{k=0}^{2m-1}\binom{2m-1}{k}\dddd{\GGGG(1+\aaaa+k)}{\GGGG(1+\aaaa)}\dddd{y^{(2m-1-k)}(0)}{x^{1+\aaaa+k}}\rrrr)h^{2m}.\nonumber
\end{align}
From \eqref{27_10} and Lemma 11 we obtain \eqref{27_20}.
	\begin{lem} Asymptotic expansion formula
	\begin{align}\label{27_20}
\dddd{1}{h^\aaaa}\sum_{k=1}^{n-1}\dddd{y(x-kh)}{k^{1+\aaaa}}=&\GGGG(-\aaaa)y^{(\aaaa)}(x)+\dddd{ y(0)}{\aaaa x^{\aaaa}}-\dddd{y(0) h}{2 x^{1+\aaaa}}\\
&+\sum_{m=0}^\infty\dddd{(-1)^m}{m!}\zzzz(1+\aaaa-m)y^{(m)}(x)h^{m-\aaaa}\nonumber\\
-\sum_{m=1}^\infty \dddd{B_{2m}}{(2m)!}&\llll( \sum_{k=0}^{2m-1}\binom{2m-1}{k}\dddd{\GGGG(1+\aaaa+k)}{\GGGG(1+\aaaa)}\dddd{y^{(2m-1-k)}(0)}{x^{1+\aaaa+k}}\rrrr)h^{2m}.\nonumber
\end{align}
\end{lem}
From \eqref{27_20} we obtain the fourth-order approximation
\begin{align}\label{27_30}
\GGGG(-\aaaa)y^{(\aaaa)}(x)=&\dddd{1}{h^\aaaa}\sum_{k=1}^{n-1}\dddd{y(x-kh)}{k^{1+\aaaa}}-\zzzz(1+\aaaa)\dddd{y(x)}{h^\aaaa}+\dddd{y(0)}{\aaaa x^\aaaa}
+\dddd{y(0)h}{2 x^{1+\aaaa}}\\
+\zzzz(\aaaa)&y'(x)h^{1-\aaaa}-\dddd{\zzzz(\aaaa-1)}{2}y''(x)h^{2-\aaaa}+\dddd{\zzzz(\aaaa-2)}{6}y'''(x)h^{3-\aaaa}\nonumber\\
-&\dddd{\zzzz(\aaaa-3)}{24}y^{(4)}(x)h^{4-\aaaa}+\dddd{ x y'(0)+(1+\aaaa)y(0)}{12x^{2+\aaaa}}h^2+O\llll( h^4 \rrrr).\nonumber
\end{align}
{\renewcommand{\arraystretch}{1.1}
\begin{table}[ht]
	\caption{Maximum error and order of approximation \eqref{27_30} for the functions $\tan^{-1} t,x=1,\; \ln (t+1),x=2$ and $\zzzz (t+2),x=3$ with $\aaaa=0.4$.}
	\small
	\centering
  \begin{tabular}{ l | c  c | c  c | c  c }
		\hline
		\hline
		\multirow{2}*{ $\quad \boldsymbol{h}$}  & \multicolumn{2}{c|}{$\boldsymbol{\tan^{-1} t}$} & \multicolumn{2}{c|}{$\boldsymbol{\ln (t+1)}$}  & \multicolumn{2}{c}{$\boldsymbol{\zzzz (t+2)}$} \\
		\cline{2-7}  
   & $Error$ & $Order$  & $Error$ & $Order$  & $Error$ & $Order$ \\ 
		\hline \hline
$0.025$    &  $4.4\times 10^{-9}$  & $3.999$  & $5.0\times 10^{-10}$  & $3.999$    & $4.2\times 10^{-10}$   & $3.999$       \\ 
$0.0125$   &  $2.7\times 10^{-10}$ & $3.999$  & $3.1\times 10^{-11}$  & $4.000$    & $2.6\times 10^{-11}$   & $3.999$       \\ 
$0.00625$  &  $1.7\times 10^{-11}$ & $3.999$  & $2.0\times 10^{-12}$  & $3.969$    & $1.6\times 10^{-12}$   & $4.003$        \\ 
$0.003125$ &  $1.1\times 10^{-12}$ & $3.961$  & $1.5\times 10^{-13}$  & $3.734$    & $1.2\times 10^{-13}$   & $3.750$        \\
\hline
  \end{tabular}
	\end{table}
	}
\subsection{Approximation for the Caputo Derivative of Order $\boldsymbol{2-\aaaa}$}
	By approximating  $y'_n$ in \eqref{25_10} using first-order backward difference approximation 
	$y'_n=(y_n-y_{n-1})/h+O(h)$ we obtain the approximation 
\begin{align}\label{28_10}
\mathcal{A}_n^{\bar{\ssss}}[y]=\dddd{1}{\GGGG(-\aaaa)h^\aaaa}\sum_{k=0}^{n}\bar{\ssss}_k^{(\aaaa)} y_{n-k}= y^{(\aaaa)}_n+O\llll(h^{2-\aaaa}\rrrr),
\end{align}
where
$$\bar{\ssss}_0^{(\aaaa)}=\zzzz(\aaaa)-\zzzz(1+\aaaa),\;\bar{\ssss}_1^{(\aaaa)}=1-\zzzz(\aaaa),$$
$$\bar{\ssss}_k^{(\aaaa)}=\dddd{1}{k^{1+\aaaa}}\quad (k=2,\cdots,n-1),\quad \bar{\ssss}_{n}^{(\aaaa)}=\zzzz(1+\aaaa)-\sum_{k=1}^{n-1}\dddd{1}{k^{1+\aaaa}}.$$
Approximation \eqref{28_10} satisfies $\mathcal{A}^{\bar{\ssss}}_n[1]=0$. The accuracy of approximation \eqref{28_10} is $O\llll(h^{2-\aaaa}\rrrr)$ when the function $y$ satisfies $y'(0)=0$.
Denote
$$W_n[\aaaa]=n S_n[\aaaa+1]-S_n[\aaaa]+\dddd{n^{1-\aaaa}}{\aaaa(1-\aaaa)}.$$
\begin{clm} Let $y(x)=x$. Then
\begin{align*}
\mathcal{A}^{\bar{\ssss}}_n[x]-y^{(\aaaa)}(x)=\dddd{W_n[\aaaa]h^{1-\aaaa}}{\GGGG(-\aaaa)}.
\end{align*}
\end{clm}
\begin{proof}
\begin{align*}
\GGGG(-\aaaa&)h^\aaaa\mathcal{A}^{\bar{\ssss}}_n[x]=\sum_{k=0}^{n}\bar{\ssss}_k^{(\aaaa)} (x-kh)=-h\sum_{k=1}^{n}k\bar{\ssss}_k^{(\aaaa)}\\
&=-h\llll(\sum_{k=1}^{n-1}k\bar{\ssss}_k^{(\aaaa)}+n \bar{\ssss}_n^{(\aaaa)} \rrrr)=h\llll(\zzzz(\aaaa)-n\oooo_n^{(\aaaa)}-\sum_{k=1}^{n-1}k\oooo_k^{(\aaaa)}\rrrr),
\end{align*}
$$\mathcal{A}^{\bar{\ssss}}_n[x]=\dddd{h^{1-\aaaa}}{\GGGG(-\aaaa)}\llll(\llll(\zzzz(\aaaa)-\sum_{k=1}^{n-1}\dddd{1}{k^{\aaaa}}  \rrrr)-n\llll(\zzzz(1+\aaaa)-\sum_{k=1}^{n-1}\dddd{1}{k^{1+\aaaa}} \rrrr)\rrrr).$$
Then
\begin{align*}
\mathcal{A}^{\bar{\ssss}}_n[y]-y^{(\aaaa)}(x)&=\dddd{h^{1-\aaaa}}{\GGGG(-\aaaa)}(n S_n(\aaaa+1)-S_n(\aaaa))-\dddd{x^{1-\aaaa}}{\GGGG(2-\aaaa)}\\
&=\dddd{h^{1-\aaaa}}{\GGGG(-\aaaa)}\llll(n S_n(\aaaa+1)-S_n(\aaaa)+\dddd{n^{1-\aaaa}}{\aaaa(1-\aaaa)}\rrrr).
\end{align*}
\end{proof}
Represent the function $y(x)$ as
$$y(x)=y(x)-y(0)-y'(0)x+y(0)+y'(0)x=z(x)+y(0)+y'(0)x.$$
The function $z(x)=y(x)-y(0)-y'(0)x$ satisfies $z(0)=z'(0)=0$.
$$\mathcal{A}^{\bar{\ssss}}_n[z]=z^{(\aaaa)}(x)+O\llll( h^{2-\aaaa} \rrrr)=y^{(\aaaa)}(x)-y'(0)\dddd{x^{1-\aaaa}}{\GGGG(2-\aaaa)}+O\llll( h^{2-\aaaa} \rrrr).$$
Then
$$\mathcal{A}^{\bar{\ssss}}_n[y]=\mathcal{A}^{\bar{\ssss}}_n[z]+y(0)\mathcal{A}^{\bar{\ssss}}_n[1]+y'(0)\mathcal{A}^{\bar{\ssss}}_n[x],$$
$$\mathcal{A}^{\bar{\ssss}}_n[y]=y^{(\aaaa)}(x)+y'(0)\mathcal{A}^{\bar{\ssss}}_n[x]-y'(0)\dddd{x^{1-\aaaa}}{\GGGG(2-\aaaa)}+O\llll( h^{2-\aaaa} \rrrr).$$
From Claim 14
$$\dddd{1}{\GGGG(-\aaaa)h^\aaaa}\sum_{k=0}^{n}\bar{\ssss}_k^{(\aaaa)} y_{n-k}= y^{(\aaaa)}_n+\llll(\dddd{y'(0) W_n[\aaaa]}{\GGGG(-\aaaa)}\rrrr)h^{1-\aaaa}+O\llll(h^{2-\aaaa}  \rrrr)$$
By approximating $y'_0$ using first-order forward difference approximation 
$$y'_0=\dddd{y_1-y_0}{h}+O(h)$$
we obtain the approximation for the Caputo derivative of order $2-\aaaa$
\begin{align}\label{29_10}
\dddd{1}{\GGGG(-\aaaa)h^\aaaa}\sum_{k=0}^{n}\ssss_k^{(\aaaa)} y_{n-k}= y_n^{(\aaaa)}+O\llll(h^{2-\aaaa}\rrrr),
\end{align}
where $\ssss_0^{(\aaaa)}=\zzzz(\aaaa)-\zzzz(1+\aaaa),\;\ssss_1^{(\aaaa)}=1-\zzzz(\aaaa),$
$$\ssss_k^{(\aaaa)}=\dddd{1}{k^{1+\aaaa}}\quad (k=2,\cdots,n-2),$$
$$\ssss_{n-1}^{(\aaaa)}=\dddd{1}{(n-1)^{1+\aaaa}}-\dddd{n^{1-\aaaa}}{\aaaa(1-\aaaa)}-n S_n[1+\aaaa]+S_n[\aaaa],$$
$$\ssss_{n}^{(\aaaa)}=(n-1)S_n[1+\aaaa]+S_n[\aaaa]+\dddd{n^{1-\aaaa}}{\aaaa(1-\aaaa)}.$$
When $n=2$ approximation \eqref{29_10} has weights $\ssss_0^{(\aaaa)}=\zzzz(\aaaa)-\zzzz(1+\aaaa)$ and
$$\ssss_1^{(\aaaa)}=2\zzzz(1+\aaaa)-2\zzzz(\aaaa)-\dddd{2^{1-\aaaa}}{\aaaa(1-\aaaa)},\; \ssss_2^{(\aaaa)}=\zzzz(\aaaa)-\zzzz(1+\aaaa)+\dddd{2^{1-\aaaa}}{\aaaa(1-\aaaa)}.$$
From the formula for sum of powers \eqref{14_05}
$$S_n[\aaaa]=\dddd{n^{1-\aaaa}}{1-\aaaa}\sum_{m=0}^{\infty}\binom{1-\aaaa}{m}\dddd{B_m}{n^m},$$
The Bernoulli number $B_1=-1/2$ and $B_{2m+1}=0$ for $m\geq 1$. Then
\begin{align}\label{30_10}
\dddd{n^{1-\aaaa}}{1-\aaaa}\sum_{m=1}^{\infty}\binom{1-\aaaa}{2m}\dddd{B_{2m}}{n^{2m}}=S_n[\aaaa]-\dddd{n^{1-\aaaa}}{1-\aaaa}+\dddd{1}{2n^\aaaa}.
\end{align}
From \eqref{30_10}
$$S_n[\aaaa]=\dddd{n^{1-\aaaa}}{1-\aaaa}-\dddd{1}{2n^\aaaa}+O\llll( \dddd{1}{n^{1+\aaaa}}\rrrr),\; S_n[\aaaa+1]=-\dddd{n^{-\aaaa}}{\aaaa}-\dddd{1}{2n^{1+\aaaa}}+O\llll( \dddd{1}{n^{2+\aaaa}}\rrrr).
$$
Then
$$\tilde{\ssss}_n^{(\aaaa)}= (1-n)\llll(\dddd{1}{\aaaa n^\aaaa} +\dddd{1}{2 n^{1+\aaaa}}\rrrr)-\llll(\dddd{n^{1-\aaaa}}{1-\aaaa}-\dddd{1}{2n^\aaaa} \rrrr)+\dddd{n^{1-\aaaa}}{\aaaa(1-\aaaa)}+O\llll(\dddd{1}{n^{1+\aaaa}} \rrrr),$$
$$\ssss_n^{(\aaaa)}= \dddd{1}{\aaaa n^\aaaa} +\llll(\dddd{n^{1-\aaaa}}{\aaaa(1-\aaaa)}-\dddd{n^{1-\aaaa}}{\aaaa} -\dddd{n^{1-\aaaa}}{1-\aaaa}\rrrr)+O\llll(\dddd{1}{n^{1+\aaaa}} \rrrr)=\dddd{1}{\aaaa n^\aaaa}+O\llll(\dddd{1}{n^{1+\aaaa}} \rrrr).$$
The weight $\ssss_n^{(\aaaa)}$ of approximation \eqref{29_10} is approximately $1/(\aaaa n^\aaaa)$. Approximation \eqref{29_10} has order $2-\aaaa$ and its weights satisfy \eqref{3_20}. In Table 9 we compute the error and the order of the numerical solution of Equation I, Equation II and Equation III which uses approximation \eqref{29_10} for the Caputo derivative. 
In all experiments in Table 1, Table 3, Table 5 and Table 9 the errors of numerical solution $NS[12]$ are smaller than the errors of numerical solutions $NS[1]$ and $NS[9]$.
{ \renewcommand{\arraystretch}{1.1}
\begin{table}[ht]
	\caption{Error and order of numerical solution $NS[40]$ of Equation I with $\aaaa=0.25$, Equation II with $\aaaa=0.5$ and Equation III with $\aaaa=0.75$.}
	\small
	\centering
  \begin{tabular}{ l | c  c | c  c | c  c }
		\hline
		\hline
		\multirow{2}*{ $\quad \boldsymbol{h}$}  & \multicolumn{2}{c|}{{\bf Equation I}} & \multicolumn{2}{c|}{{\bf Equation II}}  & \multicolumn{2}{c}{{\bf Equation III}} \\
		\cline{2-7}  
   & $Error$ & $Order$  & $Error$ & $Order$  & $Error$ & $Order$ \\ 
		\hline \hline
$0.003125$    & $0.0000263$          & $1.7472$  & $0.0000395122$        & $1.4982$    & $0.0022174$  & $1.2479$       \\ 
$0.0015625$   & $7.8\times 10^{-6}$  & $1.7486$  & $0.0000139784$        & $1.4991$    & $0.0009328$  & $1.2492$       \\ 
$0.00078125$  & $2.3\times 10^{-6}$  & $1.7493$  & $4.9\times 10^{-6}$   & $1.4995$    & $0.0003923$  & $1.2497$        \\ 
$0.000390625$ & $6.9\times 10^{-7}$  & $1.7497$  & $1.7\times 10^{-6}$   & $1.4998$    & $0.0001649$  & $1.2499$        \\
\hline
  \end{tabular}
	\end{table}
	}
\subsection{Approximation for the Caputo Derivative of Order $\boldsymbol{3-\aaaa}$}
From \eqref{27_20}, approximation \eqref{25_20} has asymptotic expansion
	\begin{align}\label{31_10}
\dddd{1}{h^\aaaa}\sum_{k=0}^{n}\oooo_k^{(\aaaa)} y(x-kh&)=\GGGG(-\aaaa)y^{(\aaaa)}(x)+\sum_{m=1}^\infty\dddd{(-1)^m}{m!}\zzzz(1+\aaaa-m)y^{(m)}(x)h^{m-\aaaa}\nonumber\\
-\sum_{m=1}^\infty \dddd{B_{2m}}{(2m)!}&\llll( \sum_{k=0}^{2m-2}\binom{2m-1}{k}\dddd{\GGGG(1+\aaaa+k)}{\GGGG(1+\aaaa)}\dddd{y^{(2m-1-k)}(0)}{x^{1+\aaaa+k}}\rrrr)h^{2m}
\end{align}
In order to construct approximations for the Caputo derivative we express expansion formula \eqref{31_10} in the following form
	\begin{align*}
\dddd{1}{h^\aaaa}\sum_{k=0}^{n}\oooo_k^{(\aaaa)} y(x-k h)=&\GGGG(-\aaaa)y^{(\aaaa)}(x)+\sum_{m=1}^\infty K_m(\aaaa)y^{(m)}(0) h^{m-\aaaa}\\
&+\sum_{m=1}^\infty\dddd{(-1)^m}{m!}\zzzz(1+\aaaa-m)y^{(m)}(x)h^{m-\aaaa}.
\end{align*}
In the previous section we showed that
$$K_1(\aaaa)=n S_n[1+\aaaa]-S_n[\aaaa]+\dddd{n^{1-\aaaa}}{\aaaa(1-\aaaa)}.$$
In this section we obtain the formula \eqref{33_10} for the coefficient $K_2(\aaaa)$ and an approximation \eqref{34_10} for the Caputo derivative of order $3-\aaaa$. By changing the order of summation in \eqref{31_10} 
	\begin{align*}
\dddd{1}{h^\aaaa}\sum_{k=0}^{n}\oooo_k^{(\aaaa)} y(x-kh)=&\GGGG(-\aaaa)y^{(\aaaa)}(x)+\sum_{m=1}^\infty\dddd{(-1)^m}{m!}\zzzz(1+\aaaa-m)y^{(m)}(x)h^{m-\aaaa}\\
-\sum_{p=1}^\infty &\llll( \sum_{m=\left\lceil  p/2\right\rceil}^{\infty}\binom{2m-1}{p}\dddd{\GGGG(2m+\aaaa-p)}{\GGGG(1+\aaaa)x^{2m+\aaaa-p}}\dddd{B_{2m}}{(2m)!}h^{2m}\rrrr)y^{(p)}(0).
\end{align*}
We have that $x^{2m+\aaaa-p}=n^{2m+\aaaa-p}h^{2m+\aaaa-p}$. Then
$$K_p(\aaaa)= -\sum_{m=\left\lceil  p/2\right\rceil}^{\infty}\binom{2m-1}{p}\dddd{\GGGG(2m+\aaaa-p)}{\GGGG(1+\aaaa)n^{2m+\aaaa-p}}\dddd{B_{2m}}{(2m)!}.$$
When $p=2$ we obtain
$$K_2(\aaaa)=-\dddd{n^{2-\aaaa}}{2} \sum_{m=1}^{\infty}(2m-1)(2m-2)\dddd{\GGGG(2m+\aaaa-2)}{\GGGG(1+\aaaa)(2m)!}\dddd{B_{2m}}{n^{2m}}.$$
From the identity
$$(2m-1)(2m-2)=(2m-2+\aaaa)(2m-1+\aaaa)-2\aaaa(2m-2+\aaaa)+\aaaa(\aaaa-1)$$
we obtain
\begin{align*}
K_2(\aaaa)=&-\dddd{n^{2-\aaaa}}{2} \sum_{m=1}^{\infty}\dddd{(1+\aaaa)(2+\aaaa)\cdots (2m-1+\aaaa)}{(2m)!}\dddd{B_{2m}}{n^{2m}}\\
&+n^{2-\aaaa} \sum_{m=1}^{\infty}\dddd{\aaaa(1+\aaaa)(2+\aaaa)\cdots (2m-2+\aaaa)}{(2m)!}\dddd{B_{2m}}{n^{2m}}\\
&-\dddd{n^{2-\aaaa}}{2} \sum_{m=1}^{\infty}\dddd{(\aaaa-1)a(1+\aaaa)(2+\aaaa)\cdots (2m-3+\aaaa)}{(2m)!}\dddd{B_{2m}}{n^{2m}},
\end{align*}
\begin{align*}
K_2(\aaaa)=&- \dddd{n^{2-\aaaa}}{2\aaaa}\sum_{m=1}^{\infty}\dddd{\aaaa(1+\aaaa)(2+\aaaa)\cdots (2m-1+\aaaa)}{(2m)!}\dddd{B_{2m}}{n^{2m}}\\
&+\dddd{n^{2-\aaaa}}{\aaaa-1} \sum_{m=1}^{\infty}\dddd{(\aaaa-1)\aaaa(1+\aaaa)(2+\aaaa)\cdots (2m-2+\aaaa)}{(2m)!}\dddd{B_{2m}}{n^{2m}}\\
&-\dddd{n^{2-\aaaa}}{2(\aaaa-2)} \sum_{m=1}^{\infty}\dddd{(\aaaa-2)(\aaaa-1)a(1+\aaaa)(2+\aaaa)\cdots (2m-3+\aaaa)}{(2m)!}\dddd{B_{2m}}{n^{2m}},
\end{align*}
\begin{align*}
K_2(\aaaa)=&\dddd{n^{2}}{2}\llll[ \dddd{n^{-\aaaa}}{-\aaaa}\sum_{m=1}^{\infty}\binom{-\aaaa}{2m}\dddd{B_{2m}}{n^{2m}}\rrrr]-n\llll[\dddd{n^{1-\aaaa}}{1-\aaaa} \sum_{m=1}^{\infty}\binom{1-\aaaa}{2m}\dddd{B_{2m}}{n^{2m}}\rrrr]\\
&+\dddd{1}{2}\llll[\dddd{n^{2-\aaaa}}{2-\aaaa} \sum_{m=1}^{\infty}\binom{2-\aaaa}{2m}\dddd{B_{2m}}{n^{2m}}\rrrr].
\end{align*}
From \eqref{30_10}
$$\dddd{n^{-\aaaa}}{-\aaaa}\sum_{m=1}^{\infty}\binom{-\aaaa}{2m}\dddd{B_{2m}}{n^{2m}}=S_n[\aaaa+1]+\dddd{n^{-\aaaa}}{\aaaa}+\dddd{1}{2n^{1+\aaaa}},$$
$$\dddd{n^{2-\aaaa}}{2-\aaaa}\sum_{m=1}^{\infty}\binom{2-\aaaa}{2m}\dddd{B_{2m}}{n^{2m}}=S_n[\aaaa-1]-\dddd{n^{2-\aaaa}}{2-\aaaa}+\dddd{1}{2n^{\aaaa-1}}.$$
Then
\begin{align*}
K_2(\aaaa)=&\dddd{n^{2}}{2}\llll[S_n[\aaaa+1]+\dddd{n^{-\aaaa}}{\aaaa}+\dddd{1}{2n^{1+\aaaa}}\rrrr]-n\llll[S_n[\aaaa]-\dddd{n^{1-\aaaa}}{1-\aaaa}+\dddd{1}{2n^\aaaa}\rrrr]\\
&+\dddd{1}{2}\llll[S_n[\aaaa-1]-\dddd{n^{2-\aaaa}}{2-\aaaa}+\dddd{1}{2n^{\aaaa-1}}\rrrr],
\end{align*}
\begin{align}\label{33_10}
K_2(\aaaa)=\dddd{n^{2}}{2}S_n[\aaaa+1]-n S_n[\aaaa]+\dddd{1}{2}S_n[\aaaa-1]+\dddd{n^{2-\aaaa}}{\aaaa(\aaaa-1)(\aaaa-2)}.
\end{align}
Approximation \eqref{25_20} has expansion of order $3-\aaaa$
	\begin{align}\label{33_20}
\dddd{1}{h^\aaaa}\sum_{k=0}^{n}\oooo_k^{(\aaaa)} y_{n-k}=\GGGG(-\aaaa)y^{(\aaaa)}_n+&\llll( K_1(\aaaa) y'_0 -\zzzz(\aaaa) y'_n\rrrr) h^{1-\aaaa}\\
&+\llll( K_2(\aaaa) y''_0 -\dddd{1}{2}\zzzz(\aaaa) y''_n\rrrr) h^{2-\aaaa}+O\llll(h^{3-\aaaa} \rrrr).\nonumber
\end{align}
By  approximating $y'_0,y''_0,y'_n,y''_n$ in \eqref{33_20} using the approximations
$$y'_0=\dddd{1}{h}\llll(-\dddd{3}{2}y_0+2y_1-\dddd{1}{2}y_2 \rrrr)+O\llll(h^2\rrrr),\; y''_0=\dddd{1}{h^2}\llll(y_0-2y_1+y_2 \rrrr)+O\llll(h^2\rrrr),$$
$$y'_n=\dddd{1}{h}\llll(\dddd{3}{2}y_n-2y_{n-1}+\dddd{1}{2}y_{n-2}\rrrr)+O\llll(h^2\rrrr),\; y''_n=\dddd{1}{h^2}\llll(y_n-2y_{n-1}+y_{n-2}\rrrr)+O\llll(h\rrrr),$$
we obtain the approximation for the Caputo derivative of order $3-\aaaa$
\begin{align}\label{34_10}
\dddd{1}{\GGGG(-\aaaa)h^\aaaa}\sum_{k=0}^{n}\dddddd_k^{(\aaaa)} y_{n-k}= y_n^{(\aaaa)}+O\llll(h^{3-\aaaa}\rrrr),
\end{align}
where
$$\dddddd_0^{(\aaaa)}=-\zzzz(1+\aaaa)+\dddd{3}{2}\zzzz(\aaaa)-\dddd{1}{2}\zzzz(\aaaa-1),\;\dddddd_1^{(\aaaa)}=1-2\zzzz(\aaaa)+\zzzz(\aaaa-1),$$
$$\dddddd_2^{(\aaaa)}=\dddd{1}{2^{1+\aaaa}}+\dddd{1}{2}\zzzz(\aaaa)-\dddd{1}{2}\zzzz(\aaaa-1),\;\dddddd_k^{(\aaaa)}=\dddd{1}{k^{1+\aaaa}}\quad (k=3,\cdots,n-3),$$
$$\dddddd_{n-2}^{(\aaaa)}=\oooo_{n-2}^{(\aaaa)}+\dddd{1}{2}K_1(\aaaa)-K_2(\aaaa),\;\dddddd_{n-1}^{(\aaaa)}=\oooo_{n-1}^{(\aaaa)}-2K_1(\aaaa)+2K_2(\aaaa),$$
$$\dddddd_n^{(\aaaa)}=\oooo_n^{(\aaaa)}+\dddd{3}{2}K_1(\aaaa)-K_2(\aaaa).$$
When $n=2$, approximation \eqref{34_10} has weights
$$
\dddddd_0^{(\aaaa)}=\dddd{ (\aaaa+2)}{2^{\aaaa}(2-\aaaa) (\aaaa-1)},\; \dddddd_1^{(\aaaa)}=\dddd{4}{2^\aaaa (2-\aaaa) (1-\aaaa)},\; \dddddd_2^{(\aaaa)}=\dddd{ 2-3 \aaaa}{2^{\aaaa}(2-\aaaa) (1-\aaaa) \aaaa}.$$
When $n=3$: $\dddddd_0^{(\aaaa)}=-\frac{1}{2}\zeta (\aaaa-1)+\frac{3 }{2}\zeta (\aaaa)-\zeta (\aaaa+1),$
$$\dddddd_1^{(\aaaa)}=\dddd{3 }{2}\zeta (\aaaa-1)+3 \zeta (\aaaa+1)-\dddd{9 }{2}\zeta (\aaaa)-\dddd{3^{1-\aaaa} (\aaaa+4)}{2 \aaaa (1-\aaaa) (2-\aaaa)},$$
$$ \dddddd_2^{(\aaaa)}=-\dddd{3}{2}\zeta (\aaaa-1)+\dddd{9 }{2}\zeta (\aaaa)-3 \zeta (\aaaa+1)+\dddd{2\times 3^{1-\aaaa} (\aaaa+1)}{\aaaa(1-\aaaa) (2-\aaaa)},$$
$$\dddddd_3^{(\aaaa)}=\dddd{1}{2}\zeta (\aaaa-1)+\zeta (\aaaa+1)-\dddd{3}{2}\zeta (\aaaa)-\dddd{3^{2-\aaaa}}{2 (1-\aaaa) (2-\aaaa)}.$$
In the previous sections computed the numerical solutions of the fractional relaxation equation of order $1-\aaaa,2-\aaaa$ and $2$ using the second-order approximation \eqref{19_20} for the value $y(h)$. In order to obtain a third-order approximation for  $y(h)$ we need to compute the values of $y'(0)$ and $y''(0)$. 
By applying fractional differentiation of order $1-\aaaa$ to equation \eqref{19_10} we obtain
\begin{align}\label{34_20}
y'(x)+y^{(1-\aaaa)}(x)=F^{(1-\aaaa)}(x).
\end{align}
Then $y'(0)=F^{(1-\aaaa)}(0)$. In \cite{Dimitrov2016} we showed that
$$\dddd{d}{dx}y^{(1-\aaaa)}(x)=y^{(2-\aaaa)}(x)+\dddd{y'(0)x^{\aaaa-1}}{\GGGG(\aaaa)}.$$
By differentiating \eqref{34_20} we obtain
$$y''(x)+\dddd{d}{dx}y^{(1-\aaaa)}(x)=\dddd{d}{dx}F^{(1-\aaaa)}(x),$$
$$y''(x)+y^{(2-\aaaa)}(x)+\dddd{y'(0)x^{\aaaa-1}}{\GGGG(\aaaa)}=\dddd{d}{dx}F^{(1-\aaaa)}(x).$$
When the solution of equation \eqref{19_10} has a bounded second derivative, the value of $y''(0)$ is computed with the formula
$$y''(0)=\lim_{x\rightarrow 0}\llll(\dddd{d}{dx}F^{(1-\aaaa)}(x)-\dddd{y'(0)x^{\aaaa-1}}{\GGGG(\aaaa)}\rrrr).$$
Now we compute the values of the derivatives $y'(0),y''(0)$ of the exact solutions of Equation I, Equation II and Equation III.
$$F_1(x)=\sum_{k=0}^4 x^k+\sum_{k=1}^4 \dddd{k! x^{k-\aaaa}}{\GGGG(k+1-\aaaa)},$$
$$F_1^{(1-\aaaa)}(x)=\sum_{k=0}^3 \dddd{x^{k+\aaaa}}{\GGGG(k+1-\aaaa)}+\sum_{k=0}^3 (k+1) x^k.$$
Then $y'(0)=F_1^{(1-\aaaa)}(0)=1$.
$$\dddd{d }{d x}F_1^{(1-\aaaa)}(x)=\sum_{k=0}^3 \dddd{x^{k+\aaaa-1}}{\GGGG(k-\aaaa)}+\sum_{k=1}^3 k(k+1) x^{k-1}.$$
Then
\begin{align*}
y''(0)=\lim_{x\rightarrow 0}\Bigg(\dddd{d}{dx}&F_1^{(1-\aaaa)}(x)-\dddd{x^{\aaaa-1}}{\GGGG(\aaaa)}\Bigg)\\
&=\lim_{x\rightarrow 0}\sum_{k=1}^3 \llll(\dddd{x^{k+\aaaa-1}}{\GGGG(k-\aaaa)}+ k(k+1) x^{k-1}\rrrr)=2
\end{align*}
The exact solution of Equation I satisfies $y'(0)=1$ and  $y''(0)=2$.
$$F_{2}(x)=e^x+x^{1-\aaaa}E_{1,2-\aaaa}(x)=\sum_{k=0}^\infty \dddd{x^k}{k!}+\sum_{k=1}^4 \dddd{ x^{k+1-\aaaa}}{\GGGG(k+2-\aaaa)},$$
$$F_{2}^{(1-\aaaa)}(x)=\sum_{k=1}^\infty \dddd{x^{k-1+\aaaa}}{\GGGG(k+\aaaa)}+\sum_{k=0}^\infty \dddd{x^k}{k!}=\sum_{k=1}^\infty \dddd{x^{k-1+\aaaa}}{\GGGG(k+\aaaa)}+e^x.$$
Then $y'(0)=F_{2}^{(1-\aaaa)}(0)=1$.
$$\dddd{d }{dx}F_{2}^{(1-\aaaa)}(x)=\sum_{k=1}^\infty \dddd{x^{k-2+\aaaa}}{\GGGG(k+\aaaa-1)}+e^x,$$
$$y''(0)=\lim_{x\rightarrow 0}\llll(\dddd{d}{dx}F_{2}^{(1-\aaaa)}(x)-\dddd{x^{\aaaa-1}}{\GGGG(\aaaa)}\rrrr)=\lim_{x\rightarrow 0}\llll(\sum_{k=2}^\infty \dddd{x^{k-2+\aaaa}}{\GGGG(k+\aaaa-1)}+e^x\rrrr)=1.$$
The exact solution of Equation II satisfies $y'(0)=y''(0)=1$.
\begin{align*} 
F_{3}(x)=\cos (2\pi x)+i \pi x^{1-\aaaa}
\llll(E_{1,2-\aaaa}(2\pi i x)-E_{1,2-\aaaa}(-2\pi i x)\rrrr),
\end{align*}
$$F_{3}(x)=\sum_{k=0}^\infty \dddd{\llll(-4\pi^2 \rrrr)^k x^{2k}}{(2k)!}+\sum_{k=1}^\infty \dddd{\llll(-4\pi^2 \rrrr)^k x^{2k-\aaaa}}{\GGGG(2k+1-\aaaa)},$$
$$F_{3}^{(1-\aaaa)}(x)=\sum_{k=1}^\infty \llll(-4\pi^2 \rrrr)^k\llll( \dddd{ x^{2k-1+\aaaa}}{\GGGG(2k+\aaaa)}+\dddd{ x^{2k-1}}{(2k-1)!} \rrrr).$$
Then $y'(0)=F_{3}^{(1-\aaaa)}(0)=0$.
$$\dddd{d }{dx}F_{3}^{(1-\aaaa)}(x)=\sum_{k=1}^\infty \llll(-4\pi^2 \rrrr)^k\llll( \dddd{ x^{2k-2+\aaaa}}{\GGGG(2k-1+\aaaa)}+\dddd{ x^{2k-2}}{(2k-2)!} \rrrr),$$
$$y''(0)=\lim_{x\rightarrow 0}\dddd{d}{dx}F_{3}^{(1-\aaaa)}(x)=-4\pi^2.$$
The exact solution of Equation III satisfies $y'(0)=0,y''(0)=-4\pi^2$. Let
$$\tilde{y}_1=y(0)+y'(0)h+y''(0)h^2/2.$$
From Taylor's formula $\tilde{y}_1$ is an approximation for $y(h)$ with accuracy $O\llll(h^3\rrrr)$.
The accuracy of numerical solution  $NS[45]$  with initial values $u_0=y_0,u_1=\tilde{y}_1$ is $O\llll( h^{3-\aaaa}\rrrr)$. In Table 10 we compute the maximum error and the order of numerical solution $NS[45]$ of Equation I, Equation II  and Equation III. In Figure 3 we compare numerical solutions $NS[34],NS[40]$ and $NS[45]$ of Equation III and $\aaaa=0.6$.
	\begin{figure}[t!]
  \centering
  \caption{Graph of the exact solution of Equation III  and  numerical solutions $NS[12]$-red,  and $NS[13]$-blue, and $NS[34]$-green for $h=0.1$ and $\alpha=0.6$.}
  \includegraphics[width=0.55\textwidth]{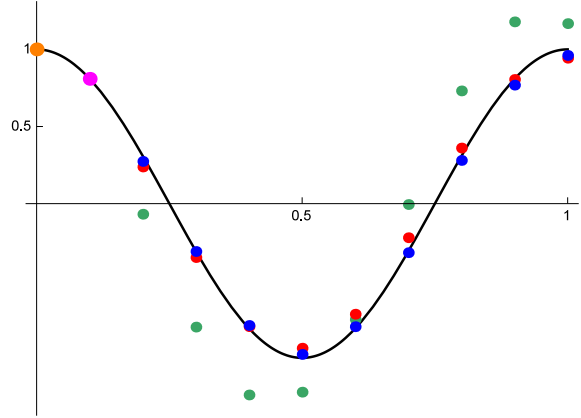}
\end{figure}
\setlength{\tabcolsep}{0.5em}
{ \renewcommand{\arraystretch}{1.1}
\begin{table}[ht]
	\caption{Error and order of numerical solution $NS[13]$ of Equation I with $\aaaa=0.25$, Equation II with $\aaaa=0.5$ and Equation III with $\aaaa=0.75$.}
	\small
	\centering
  \begin{tabular}{ l | c  c | c  c | c  c }
		\hline
		\hline
		\multirow{2}*{ $\quad \boldsymbol{h}$}  & \multicolumn{2}{c|}{{\bf Equation I}} & \multicolumn{2}{c|}{{\bf Equation II}}  & \multicolumn{2}{c}{{\bf Equation III}} \\
		\cline{2-7}  
   & $Error$ & $Order$  & $Error$ & $Order$  & $Error$ & $Order$ \\ 
		\hline \hline
$0.003125$    & $7.9\times 10^{-8}$  & $2.7458$  & $7.6\times 10^{-8}$   & $2.4945$    & $0.00003552$         & $2.2495$       \\ 
$0.0015625$   & $1.2\times 10^{-8}$  & $2.7479$  & $1.3\times 10^{-8}$   & $2.4973$    & $7.5\times 10^{-6}$  & $2.2499$       \\ 
$0.00078125$  & $1.7\times 10^{-9}$  & $2.7490$  & $2.4\times 10^{-9}$   & $2.4987$    & $1.6\times 10^{-6}$  & $2.2500$        \\ 
$0.000390625$ & $2.6\times 10^{-10}$ & $2.7495$  & $4.2\times 10^{-10}$  & $2.4993$    & $3.3\times 10^{-7}$  & $2.2500$        \\
\hline
  \end{tabular}
	\end{table}
	}
\section{Conclusions}
In the present paper we obtained approximations of  the Caputo derivative of order $2-\aaaa,2$ and $3-\aaaa$ whose weights consist of terms which have power $-\aaaa$ and $-1-\aaaa$. In all experiments the accuracy of the numerical solutions  which use approximation \eqref{6_20} for the Caputo derivative is higher than the accuracy of the numerical solutions using approximations \eqref{3_10} and \eqref{5_30}. A question for future work is to construct an approximation of order $2-\aaaa$ whose weights consist of terms which have power $-\aaaa$, where the accuracy of the numerical solutions of Equation I, Equation II and Equation III is higher than the accuracy of the numerical solutions using the $L1$ approximation \eqref{3_10}. 

\end{document}